\documentclass[12pt]{amsart}
\usepackage{amssymb}
\usepackage{verbatim}
\usepackage{hyperref}
\usepackage{color}

\begin{document}
\def\C{\mathbb{C}}
\def\R{\mathbb{R}}
\def\Rn{{\mathbb{R}^n}}
\def\Rns{{\mathbb{R}^{n+1}}}
\def\Sn{{{S}^{n-1}}}
\def\M{\mathbb{M}}
\def\N{\mathbb{N}}
\def\Q{{\mathbb{Q}}}
\def\Z{\mathbb{Z}}
\def\F{\mathcal{F}}
\def\L{\mathcal{L}}
\def\S{\mathcal{S}}
\def\supp{\operatorname{supp}}
\def\essi{\operatornamewithlimits{ess\,inf}}
\def\esss{\operatornamewithlimits{ess\,sup}}

\newtheorem{theorem}{Theorem}    
\newtheorem{proposition}[theorem]{Proposition}
\newtheorem{conjecture}[theorem]{Conjecture}
\def\theconjecture{\unskip}
\newtheorem{corollary}[theorem]{Corollary}
\newtheorem{lemma}[theorem]{Lemma}
\newtheorem{sublemma}[theorem]{Sublemma}
\newtheorem{observation}[theorem]{Observation}
\theoremstyle{definition}
\newtheorem{definition}{Definition}
\newtheorem{notation}[definition]{Notation}
\newtheorem{remark}[definition]{Remark}
\newtheorem{question}[definition]{Question}
\newtheorem{questions}[definition]{Questions}
\newtheorem{example}[definition]{Example}
\newtheorem{problem}[definition]{Problem}
\newtheorem{exercise}[definition]{Exercise}

\numberwithin{theorem}{section} \numberwithin{definition}{section}
\numberwithin{equation}{section}

\def\earrow{{\mathbf e}}
\def\rarrow{{\mathbf r}}
\def\uarrow{{\mathbf u}}
\def\varrow{{\mathbf V}}
\def\tpar{T_{\rm par}}
\def\apar{A_{\rm par}}

\def\reals{{\mathbb R}}
\def\torus{{\mathbb T}}
\def\heis{{\mathbb H}}
\def\integers{{\mathbb Z}}
\def\naturals{{\mathbb N}}
\def\complex{{\mathbb C}\/}
\def\distance{\operatorname{distance}\,}
\def\support{\operatorname{support}\,}
\def\dist{\operatorname{dist}\,}
\def\Span{\operatorname{span}\,}
\def\degree{\operatorname{degree}\,}
\def\kernel{\operatorname{kernel}\,}
\def\dim{\operatorname{dim}\,}
\def\codim{\operatorname{codim}}
\def\trace{\operatorname{trace\,}}
\def\Span{\operatorname{span}\,}
\def\dimension{\operatorname{dimension}\,}
\def\codimension{\operatorname{codimension}\,}
\def\nullspace{\scriptk}
\def\kernel{\operatorname{Ker}}
\def\ZZ{ {\mathbb Z} }
\def\p{\partial}
\def\rp{{ ^{-1} }}
\def\Re{\operatorname{Re\,} }
\def\Im{\operatorname{Im\,} }
\def\ov{\overline}
\def\eps{\varepsilon}
\def\lt{L^2}
\def\diver{\operatorname{div}}
\def\curl{\operatorname{curl}}
\def\etta{\eta}
\newcommand{\norm}[1]{ \|  #1 \|}
\def\expect{\mathbb E}
\def\bull{$\bullet$\ }
\def\xone{x_1}
\def\xtwo{x_2}
\def\xq{x_2+x_1^2}
\newcommand{\abr}[1]{ \langle  #1 \rangle}
\allowdisplaybreaks

\newcommand{\Norm}[1]{ \left\|  #1 \bigg\| }
\newcommand{\set}[1]{ \left\{ #1 \bigg\} }
\def\one{\mathbf 1}
\def\whole{\mathbf V}
\newcommand{\modulo}[2]{[#1]_{#2}}

\def\scriptf{{\mathcal F}}
\def\scriptg{{\mathcal G}}
\def\scriptm{{\mathcal M}}
\def\scriptb{{\mathcal B}}
\def\scriptc{{\mathcal C}}
\def\scriptt{{\mathcal T}}
\def\scripti{{\mathcal I}}
\def\scripte{{\mathcal E}}
\def\scriptv{{\mathcal V}}
\def\scriptw{{\mathcal W}}
\def\scriptu{{\mathcal U}}
\def\scriptS{{\mathcal S}}
\def\scripta{{\mathcal A}}
\def\scriptr{{\mathcal R}}
\def\scripto{{\mathcal O}}
\def\scripth{{\mathcal H}}
\def\scriptd{{\mathcal D}}
\def\scriptl{{\mathcal L}}
\def\scriptn{{\mathcal N}}
\def\scriptp{{\mathcal P}}
\def\scriptk{{\mathcal K}}
\def\frakv{{\mathfrak V}}

\author{Zengyan Si}
\address{
        Zengyan Si\\
       School of Mathematics and Information Science\\
       Henan Polytechnic University\\ Jiaozuo 454000\\People's
Republic of China} \email{zengyan@hpu.edu.cn}
\thanks{The first author and second author were supported by NSFC (No.11401175, No.11501169 and No. 11471041), the Fundamental Research Funds for the Central Universities (No. 2014KJJCA10) and NCET-13-0065. The third named author was supported partly by Grant-in-Aid for Scientific Research (C) Nr. 15K04942, Japan Society
for the Promotion of Science.\\
\indent Corresponding author: Qingying Xue \indent Email: qyxue@bnu.edu.cn}



\author{Qingying Xue}
\address{Qingying Xue
\\
School of Mathematical Sciences
\\
Beijing Normal University
\\
Laboratory of Mathematics and Complex Systems
\\
Ministry of Education
\\
Beijing 100875
\\
People's Republic of China } \email{qyxue@bnu.edu.cn}
\author{K\^{o}z\^{o} Yabuta}
\address{K\^{o}z\^{o} Yabuta
\\ Research Center for Mathematical Sciences
\\
Kwansei Gakuin University
\\
Gakuen 2-1, Sanda 669-1337
\\
Japan }

\date{\today}

\title
[On the bilinear square Fourier multiplier operators ...] {On the bilinear square Fourier multiplier operators and related multilinear square functions}

\maketitle
\begin{abstract}
Let $n\ge 1$ and $\mathfrak{T}_{m}$ be the bilinear square Fourier multiplier
operator associated with a symbol $m$, which is defined by
$$
\mathfrak{T}_{m}(f_1,f_2)(x)
= \biggl( \int_{0}^\infty\Big|\int_{(\mathbb{R}^n)^2}
e^{2\pi ix\cdot (\xi_1 +\xi_2) }m(t\xi_1,t\xi_2)
\hat{f}_{1}(\xi_1)\hat{f}_{2}(\xi_2)d\xi_1 d\xi_2\Big|^2\frac{dt}{t }
\biggr)^{\frac 12}.
$$
Let $s$ be an integer with $s\in[n+1,2n]$ and $p_0$ be a number satisfying
$2n/s\le p_0\le 2$. Suppose that
$\nu_{\vec{\omega}}=\prod_{i=1}^2\omega_i^{p/ p_i}$ and each $\omega_i$ is a
nonnegative function on $\mathbb{R}^n$.
In this paper, we show that $\mathfrak{T}_{m}$ is bounded from
$L^{p_1}(\omega_1)\times L^{p_2}(\omega_2)$ to $L^p(\nu_{\vec{\omega}})$
if $p_0< p_1, p_2<\infty$ with $1/p=1/p_1+ 1/p_2$. Moreover, if $p_0>2n/s$ and
$p_1=p_0$ or $p_2=p_0$, then $\mathfrak{T}_{m}$ is bounded from
$L^{p_1}(\omega_1)\times L^{p_2}(\omega_2)$ to
$L^{p,\infty}(\nu_{\vec{\omega}})$. The weighted end-point $L\log L$ type
estimate and strong estimate for the commutators of $\mathfrak{T}_{m}$ are
also given. These were done by considering the boundedness of some related
multilinear square functions associated with mild regularity kernels and essentially improving some basic lemmas which have been used before.
\end{abstract}
\section{Introduction}
\subsection{Background}
The multilinear Calder\'{o}n-Zygmund operators were first introduced and
studied by Coifman and Meyer \cite{coif1, coif2}, and later on by
Grafakos and Torres \cite{Gra1, Gra2}. Due to the close relationship between
the Calder\'{o}n-Zygmund operators and Littlewood-Paley operators,
in the meantime, the multilinear Littlewood-Paley $g$-function and related
multilinear Littlewood-Paley type estimates were used in PDE and other fields
(\cite{Coifman01,Coifman02,David,Fabes1,Fabes2,Fabes}). For example, in
\cite{Fabes}, the authors studied a class of multilinear square functions and
applied it to the well-known Kato's problem. For more works about multilinear
Littlewood-Paley type operators, see \cite{chenxue,Shi} and the references
therein. Recently, in the theory of multilinear operators, efforts have been
made to remove or replace the smoothness condition assumed on the kernels,
among these achievements are the nice works of Bui and Duong \cite{anh},
Grafakos, Liu and Yang \cite{grafakos}, Tomita \cite{NT}, Grafakos, Miyachi and Tomita \cite{grafakos44} and more recent work of Grafakos,
He and Honz\'{i}k \cite{Gra3}.

It is also well known that the following $N$-linear $(N \ge 1)$ Fourier
multiplier operator $T_m$ was introduced by Coifman and Meyer in \cite{coif3}.
\begin{equation*}
T_m(f_1,\cdots,f_N)(x)
=\frac{1}{(2\pi)^{nN}} \int_{(\mathbb{R}^n)^N} e^{ix \cdot ( \xi_1 + \cdots + \xi_N )} m(\xi) \widehat{f_1}(\xi_1) \cdots \widehat{f_N}(\xi_N)d\xi.
\end{equation*}
Suppose that $m$ is a bounded function on $\mathbb{R}^{nN}\backslash \{0\}$ and it satisfies that
\begin{equation}\label{Condition}
|\partial_{\xi_1}^{{\alpha}_1}\cdots\partial_{\xi_N}^{{\alpha}_N}
m(\xi_1,\cdots,\xi_N)|\leq
C_{\alpha}(|\xi_1|+\cdots+|\xi_m|)^{-(|\alpha_1|+\cdots+|\alpha_N|)},
\end{equation}
away from the origin for all sufficiently large multiindices $\alpha_j$. Then, it was shown in \cite{coif3}
that $T_{m}$ is bounded from $L^{p_1}(\mathbb{R}^{n})\times\cdots\times L^{p_N}(\mathbb{R}^{n})$
to $L^p(\mathbb{R}^{n})$.
In 2010, by weakening the smoothness condition $(\ref{Condition})$, Tomita \cite{NT} gave a H\"{o}rmander type theorem for $T_{m}$.
Later, Grafakos and Si \cite{si} gave a similar result for the case $p\leq 1$ by using the $L^r$-based Sobolev spaces ($1<r\leq 2$).
Subsequently, Grafakos, Miyachi and Tomita \cite{grafakos44} proved that if $m\in L^{\infty}(\mathbb{R}^{nN})$ satisfies
$\sup_{k\in\mathbb{Z}}\|m_k\|_{W^{(s_1,\cdots,s_N)}({\mathbb{R}}^{nN})}<\infty
\hspace{0.1cm}\text{with}\hspace{0.05cm} s_1,\cdots,s_N>n/2$, then $T_m$ is bounded from
$L^2({\mathbb{R}}^{n})\times L^{\infty}({\mathbb{R}}^{n})\times \cdots \times L^{\infty}({\mathbb{R}}^{n})$ to $L^2({\mathbb{R}}^{n})$.

A weighted version of the results in \cite{NT} for $T_m$ was given by Fujita and Tomita \cite{FT} under the H\"{o}rmander condition with classical $A_p$ weights. Recently,
Li and Sun \cite{lisun} demonstrated a H\"{o}rmander type multiplier theorem for $T_m$ with multiple weights. Furthermore, they obtained some weighted
estimates for the commutators of $T_m$ with vector version of $BMO$ functions. Still more recenty, Li, Xue and Yabuta \cite{lixueya} considered the estimates about weighted Carleson measure, and consequently they obtained some weighted results of $T_m$ by considering the missing endpoint parts of the results in \cite{FT}.
\subsection{Results on multilinear Fourier multiplier}
It is also well known that Lacey \cite{lacey} studied the following bilinear Littlewood-Paley square function defined by
$$
T(f,g)(x)=\bigg( \sum_{l\in \mathbb{Z}^d}|T_{\phi_l}(f,g)(x)|^2\bigg)^{1/2},
$$
where the bilinear operator $T_{\phi_l}$ associated with a smooth function $\phi_l$ whose Fourier
transform is supported in $\omega_l$ is defined by
$$
T_{\phi_l}(f,g)(x)= \int_{(\mathbb{R}^n)^2} \hat{f}(\xi)\hat{g}(\eta)\hat{\phi}_l(\xi-\eta)e^{2\pi ix\cdot (\xi+\eta)} d\xi d\eta,
$$
and $\{ \omega_l\}_{l\in \mathbb{Z}^d}$ is a sequence of disjoint cubes. The study on bilinear Littlewood-Paley square function has two motivations: One is  Alberto Calder\'on's conjectures on bilinear Hilbert transform; Another one is the norm inequalities of Littlewood-Paley type operators.

Our object
of investigation is the bilinear square Fourier multiplier operator
\begin{eqnarray}\label{squaremultiplier}
\begin{split}
  &\mathfrak{T}_{m}(f_1,f_2)(x)\\
  & = \biggl( \int_{0}^\infty \Big|\int_{(\mathbb{R}^n)^2}e^{2\pi ix\cdot (\xi_1 +\xi_2) }m(t\xi_1,t\xi_2) \hat{f}_{1}(\xi_1)\hat{f}_{2}(\xi_2)d\xi_1 d\xi_2\Big|^2\frac{dt}{t }\biggr)^{\frac 12}.
\end{split}
\end{eqnarray}
Let $K_t(x,y_1,y_2)
=\frac{1}{t^{2n}}\check{m}\big(\frac{x-y_1}{t},\frac{x-y_2}{t}\big)
$. Then, $\mathfrak{T}_{m}$ can be written in the form
 \begin{equation*}
\begin{split}
\mathfrak{T}_{m}(\vec{f})(x)=\biggl( \int_{0}^\infty  \Bigl|\int_{(\mathbb{R}^n)^2}K_t(x,y_1,y_2) f_{1}(y_1)f_{2}(y_2)dy_1dy_2\Bigr|^2\frac{dt}{t}\biggr)^{\frac 12},
\end{split}
\end{equation*}
The commutator of $\mathfrak{T}_{m}$ is defined by
\begin{equation*}
\begin{split}
\mathfrak{T}_{m}^{\vec{b}}(\vec{f})(x)=\sum_{i=1}^2\biggl( \int_{0}^\infty  \Bigl|\int_{(\mathbb{R}^n)^2}(b_i(x)-b_i(y))K_t(x,y_1,y_2) f_{1}(y_1)f_{2}(y_2)dy_1dy_2\Bigr|^2\frac{dt}{t}\biggr)^{\frac 12},
\end{split}
\end{equation*}
In this paper, we always assume that $m\in L^\infty((\mathbb{R}^n)^2)$ and satisfies the conditions
\begin{equation}\label{eq:2-lin-FM-5}
|\partial^\alpha m(\xi_1,\xi_2)|
\lesssim \frac{(|\xi_1|+|\xi_2|)^{-|\alpha|+\varepsilon_1}}
{(1+|\xi_1|+|\xi_2|)^{\varepsilon_1+\varepsilon_2}}
\end{equation}
for some $\varepsilon_1,\varepsilon_2>0$ and $|\alpha|\le s.$

The main results of this paper are:
\begin{theorem}\label{31}
 Let $s$ be an integer with $s\in[n+1,2n]$ and $p_0$ be a number satisfying
$2n/s\le p_0\le 2$.
 Let $p_0\leq p_1, p_2<\infty$, $1/p=1/p_1+ 1/p_2$,
 and $\vec \omega\in A_{\vec P/p_0}$.
Suppose that $m\in L^\infty((\mathbb{R}^n)^2)$ satisfies \eqref{eq:2-lin-FM-5}
 and that the bilinear square Fourier multiplier operator $\mathfrak{T}_{m}$
is bounded from $L^{q_1}\times  L^{q_2}$ into $L^{q,\infty}$, for
any $p_0<q_1,q_2$ and $1/q=1/q_1+1/q_2$.
Then the following weighted estimates hold.
\begin{enumerate}
 \item [(i)]If $p_1, p_2>p_0$, then
$||\mathfrak{T}_{m}(\vec{f})||_{L^p(\nu_{\vec{\omega}})}
\leq C ||f_1||_{L^{p_1}(\omega_1)}||f_2||_{L^{p_2}(\omega_2)}$.
\item [(ii)]If $p_0>2n/s$ and $p_1=p_0$ or $p_2=p_0$, then
 $$||\mathfrak{T}_{m}(\vec{f})||_{L^{p,\infty}(\nu_{\vec{\omega}})}\leq C
||f_1||_{L^{p_1}(\omega_1)}||f_2||_{L^{p_2}(\omega_2)}.$$
 \end{enumerate}
\end{theorem}
\begin{theorem}\label{34}
 Let $s$, $p_0, p_1, p_2, p$,
$\vec \omega$, $m$ and $\mathfrak{T}_{m}$ be the same as in Theorem \ref{31}.
Then the following weighted estimates hold for the commutators of $\mathfrak{T}_{m}(\vec{f})$.
\begin{enumerate}
 \item [(i)]If $p_1, p_2>p_0$, then for any $\vec{b}\in BMO^2$, it holds that
$$||\mathfrak{T}_{m}^{\vec{b}}(\vec{f})||_{L^p(\nu_{\vec{\omega}})}
\leq C||\vec{b}||_{BMO}
||f_1||_{L^{p_1}(\omega_1)}||f_2||_{L^{p_2}(\omega_2)},$$
where $||\vec{b}||_{BMO}=\max_j ||b_j||_{BMO}$.
\item [(ii)]Let $\vec{\omega}\in A_{\vec{1}}$ and $\vec{b}\in BMO^m.$ Then, there exists a
constant $C $ (depending on $\vec{b}$) such that
$$
\nu_{\vec{\omega}}\bigl( \bigl\{x\in \mathbb{R}^n: |\mathfrak{T}_{m}^{\vec{b}}(\vec{f})(x)|>t^2
\bigr\}\bigr)
\leq C \prod_{j=1}^2 \biggl( \int_{\mathbb{R}^n}
\Phi \Bigl(\frac{|f_j(x)|}{t}\Bigr)\omega_j(x)\biggr)^{1/ 2},
$$
where $\Phi(t)=t^{p_0}(1+\log^+ t)^{p_0}$.
 \end{enumerate}
\end{theorem}
The above results still hold for $m$-linear square Fourier multiplier operators. An example will be given in section $\ref{Sec-2}$, which shows that the
assumption that $\mathfrak{T}_{m}$ is bounded from $L^{q_1}\times  L^{q_2}$
into $L^{q,\infty}$ in Theorems \ref{31} and \ref{34} is reasonable.
The proofs of Theorems \ref{31} and \ref{34} will be based on the results of
multilinear square functions obtained in the next subsection.
\subsection{Results on multilinear square functions}
In order to state more known results, we need to introduce some definitions.
\begin{definition}[\textbf{Multilinear operator and multilinear
square function}]
Let $K$ be a locally integrable function defined away from the diagonal
$x=y_1=\cdots =y_m$ in $(\mathbb{R}^n)^{m+1}$ and $K_t={t^{-mn}}K(\cdot/t)$.
Then, the multilinear operator $\mathcal{T}$ and multilinear square function
$T$ are defined by
\begin{equation}\label{eq1}
\mathcal{ T}\vec{f}(x)=\int_{(\mathbb{R}^n)^m}K(x,y_1,\cdots ,y_m)f_1(y_1)
\cdots f_m(y_m)dy_1\cdots dy_m
\end{equation}
and
 \begin{equation}\label{squarefunction}
\begin{split}
T(\vec{f})(x)
=\bigg( \int_{0}^\infty \Big|\int_{(\mathbb{R}^n)^m}K_t(x,y_1,\dots,y_m)
\prod_{j=1}^mf_{j}(y_j)dy_1\dots dy_m\Big|^2\frac{dt}{t}\bigg)^{\frac 12},
\end{split}
\end{equation}
where $\vec{f}=(f_1,\dots,f_m)\in \mathcal{S}(\mathbb{R}^n)\times \cdots
\times \mathcal{S}(\mathbb{R}^n)$ and all
$x\notin \bigcap_{j=1}^m \texttt{supp} f_j$.\end{definition}
\begin{definition}[\textbf{Bui and Duong's condition}, \cite{bui}]
Let $S_j(Q)=2^jQ\setminus 2^{j-1}Q$ if $j\geq 1,$ and $S_0(Q)=Q$.
Then, assume that the following two conditions hold
\begin{enumerate}\item
[(h1)] For all $1\leq p_0\leq q_1,q_2,\dots,q_m<\infty$ and $0<q<\infty$ with
$1/q_1+\cdots+1/q_m=1/q,$ $\mathcal{T}$ maps
$L^{q_1}\times\cdots \times L^{q_m}$ into $L^{q,\infty}.$
\item
[(h2)] There exists $\delta>n/p_0$ so that for the conjugate exponent $p_0'$ of
$p_0,$ one has
\begin{eqnarray*}
&&\bigg(\int_{S_{j_m}(Q)}\cdots
\int_{S_{j_1}(Q)}\big| K(z,\vec{y})-K(x,\vec{y})\big|^{p'_0}d\vec{y}
\bigg)^{1/p'_0}\leq  C \frac{|x-z|^{m(\delta-n/p_0)}}{|Q|^{m\delta/n}}
2^{-m\delta j_0}
\end{eqnarray*}
for all ball $Q$, all $x,z\in 1/2 Q$ and $(j_1,\dots,j_m)\neq (0,\dots,0),$
where $j_0=\max_{k=1,\dots,m}\{j_k\}$.\end{enumerate}
\end{definition}
\begin{definition}[\textbf{Xue and Yan's condition}, \cite{xue}]\label{1.3}
\item For any $t\in(0,\infty),$  we assume that $K_t(x,y_1,\dots, y_m)$
satisfies the following conditions: there is a positive constant $ A>0$,
such that
\begin{equation}\label{eq2}
 \begin{split}
  \bigg( \int_{0}^\infty |K_t(z,y_1,\cdots,y_m)&-K_t(x,y_1,\cdots,y_m)|^2
  \frac{dt}{t}\bigg)^{\frac 12}
  \le\frac{A|z-x|^\gamma}{(\sum_{j=1}^m|x-y_j|)^{mn+\gamma}} ,
\end{split}
\end{equation}
whenever $|z-x|\leq \frac{1}{2} \max_{j=1}^m{|x-y_j|}$; and
\begin{equation}\label{eq3}
\begin{split}
 \bigg( \int_{0}^\infty |K_t(x,\vec{y})&-K_t(x,y_1,\dots,y'_j,\dots,y_m)|^2
 \frac{dt}{t}\bigg)^{\frac 12}
 \leq \frac{A|y_j-y_{j}'|^\gamma}{(\sum_{j=1}^m|x-y_j|)^{mn+\gamma}},
\end{split}
\end{equation}
whenever $|y_j-y_j'|\leq \frac{1}{2} \max_{j=1}^m{|x-y_j|}$;
 \begin{equation}\label{smooth}
\begin{split}
\bigg( \int_{0}^\infty |K_t(x,y_1,\cdots,y_m)|^2 \frac{dt}{t}\bigg)^{\frac 12}
\leq \frac{A}{(\sum_{j=1}^m|x-y_j|)^{mn}}.
\end{split}
\end{equation}
\end{definition}

In 2013, Bui and Duong \cite{bui} studied the boundedness of $\mathcal{T}$ on
product of weighted Lebesgue spaces with the kernel
satisfies the more weaker regularity conditions (h1) and (h2).
It should be pointed out that, under the assumptions (h1) and (h2),
the multilinear operator $\mathcal{T}$ defined in \eqref{eq1} may not fall
under the scope of the theorem of multilinear
Calder\'{o}n-Zygmund singular integral operators.
In 2015, Xue and Yan \cite{xue} established the multiple-weighted norm
inequalities for multilinear square function $T$ with kernel $K_t$ satisfies
the conditions in Definition \ref{1.3}.

Motivated by the above two works, we introduce the following new condition to
study the boundedness of multilinear square function and the associated
commutators.
\begin{definition}[\textbf{New condition}]\label{1.444}Let $1\leq p_0<\infty.$
Let $S_j(Q)=2^jQ\setminus 2^{j-1}Q$ if $j\geq 1,$ and
$S_0(Q)=Q$. Then, assume that\begin{enumerate}
\item [(H1)] For all $p_0\leq q_1,q_2,\dots,q_m<\infty$ and $0<q<\infty$ with
$1/q_1+\cdots+1/q_m=1/q,$ $T$ maps $L^{q_1}\times\cdots \times L^{q_m}$ into
$L^{q,\infty}.$
\item
[(H2)] There exists $\delta>n/p_0$ so that for the conjugate exponent $p_0'$ of
$p_0,$ one has
\begin{multline*}
\bigg(\int_{S_{j_m}(Q)}\cdots  \int_{S_{j_1}(Q)}
\bigg(\int_0^\infty\big| K_t(z,\vec{y})-K_t(x,\vec{y})\big|^2\frac{dt}{t}
\bigg)^{p'_0/ 2}d\vec{y}\bigg)^{1/p'_0}
\\
\leq  C \frac{|x-z|^{m(\delta-n/p_0)}}{|Q|^{m\delta/n}}2^{-m\delta j_0}
\end{multline*}
for all balls $Q$, all $x,z\in 1/2 Q$ and $(j_1,\dots,j_m)\neq (0,\dots,0),$
where $j_0=\max_{k=1,\dots,m}\{j_k\}$.
\item
[(H3)] There exists some positive constant $C>0$   such that
\begin{equation*}\label{smooth3-2}
\biggl(
\int_{S_{j_m}(Q)}\cdots  \int_{S_{j_1}(Q)}
\biggl( \int_{0}^\infty |K_t(x,y_1,\cdots,y_m)|^2 \frac{dt}{t}
\biggr)^{\frac {p_0'}2}d\vec y\biggr)^{1/p_0}
\le C\frac{2^{-mnj_0/ p_0}}{|Q|^{m/ p_0}}
\end{equation*}
for all balls $Q$ with center at $x$ and $(j_1,\dots,j_m)\neq (0,\dots,0),$
where $j_0=\max_{k=1,\dots,m}\{j_k\}$.\end{enumerate}
\end{definition}
\begin{definition}[\textbf{Commutators of multilinear square operator}]
The commutators of multilinear square operator $T$ with BMO functions
$\vec b=(b_1,b_2,\dots, b_m)$ are defined by
 \begin{equation}
\begin{split}
T_{\vec{b}}(\vec{f})(x)
=\sum_{i=1}^m\bigg( \int_{0}^\infty \Big|
\int_{(\mathbb{R}^n)^m}(b_{i}(x)-b_i(y_i))K_t(x,\vec{y})
\prod_{j=1}^mf_{j}(y_j)dy_1 \dots dy_m\Big|^2\frac{dt}{t}\bigg)^{\frac 12},
\end{split}
\end{equation}
for any $\vec{f}=(f_1,\dots,f_m)\in \mathcal{S}(\mathbb{R}^n)\times \cdots
\times \mathcal{S}(\mathbb{R}^n)$ and all $x\notin \bigcap_{j=1}^m
\texttt{supp} f_j.$
\end{definition}

We obtain the following weighted estimates.
\begin{theorem}\label{1}
 Let $T$ be the
multilinear square function with a kernel satisfying conditions (H1), (H2) and
(H3) for some $1\leq p_0<\infty.$
Then, for any $p_0\le p_1,\dots,p_m<\infty$, $1/p=1/p_1+\cdots +1/p_m$ and
$\vec{\omega}\in A_{\vec{P}/p_0}$, the following weighted estimates hold.
 \begin{enumerate}
   \item If there is no $p_i=p_0$, then
   $\|T(\vec{f}) \|_{L^{p}(\nu_{\vec{w}})}
   \leq C \prod_{i=1}^{m} \|f_i\|_{L^{p_i}(w_i)} .$\\
   \item If there is a $p_i=p_0$, then
   $\|T(\vec{f}) \|_{L^{p,\infty}(\nu_{\vec{w}})}
   \leq C \prod_{i=1}^{m} \|f_i\|_{L^{p_i}(w_i)}. $
 \end{enumerate}
\end{theorem}
As for the commutators of $T$, we obtain the following weighted  estimates.
\begin{theorem}\label{2}
 Let $T$ be the
multilinear square function with a kernel satisfying conditions (H1), (H2) and (H3) for some $1\leq p_0<\infty$. Let $\vec{b}\in BMO^m$.
Then, for any $p_0<p_1,\dots,p_m<\infty$, $1/p=1/p_1+\cdots +1/p_m$ and
$\vec{\omega}\in A_{\vec{P}/p_0}$,
we have
$$||T_{\vec{b}}\vec{f}||_{L^p(\nu_{\vec{\omega}})}\leq C||\vec{b}||_{BMO} \prod
\limits_{i=1}^{m}||f_i||_{L^{p_i}(\omega_i)},$$
where $||\vec{b}||_{BMO}=\max_j ||b_j||_{BMO}.$\end{theorem}
\begin{theorem}\label{3}
 Let $T$ be the
multilinear square function with a kernel satisfying conditions (H1), (H2) and (H3) for some $1\leq p_0<\infty$. Let $\vec{b}\in BMO^m$. Let $\vec{\omega}\in A_{\vec{1}}$ and $\vec{b}\in BMO^m.$ Then, there exists a
constant $C $ (depending on $\vec{b}$) such that $$
\nu_{\vec{\omega}}\bigl( \bigl\{x\in \mathbb{R}^n: |T_{\vec{b}}\vec{f}(x)|>t^m
\bigr\}\bigr)
\leq C \prod_{j=1}^m \biggl( \int_{\mathbb{R}^n}
\Phi \Bigl(\frac{|f_j(x)|}{t}\Bigr)\omega_j(x)\biggr)^{1/ m},
$$
where $\Phi(t)=t^{p_0}(1+\log^+ t)^{p_0}$.\end{theorem}

We organize this paper as follows: Section $\ref{Sec-2}$ contains one example concerning with the new assumption on $\mathfrak{T}_{m}$. Section $\ref{Sec-3}$ will be devoted to establish two key propositions related to multilinear square Fourier multiplier operator, which can be used to prove Theorem \ref{31}-\ref{34}. In section $\ref{SEC4}$, we will give the proofs of Theorem \ref{1} and Theorem \ref{2}.  Section $\ref{Sec-5}$ will be devoted to give the proof of Theorem \ref{3}.

Throughout this paper, the notation $A \lesssim B$ stands for $A \leq C B$ for some positive constant C independent of A and B.
\section{An example}\label{Sec-2}
In this section, an example will be given to show that there are some multilinear square Fourier multiplier operators which are bounded from $L^{q_1}(\mathbb{R}^n)\times L^{q_2}(\mathbb{R}^n)$ to $L^{q}(\mathbb{R}^n)$. Thus, the assumption that $\mathfrak{T}_{m}$ is bounded from $L^{q_1}\times  L^{q_2}$ into $L^{q,\infty}$ in Theorem \ref{31}- \ref{34} is reasonable.

Let \begin{align*}
&\widetilde{\mathcal T}_m(\vec f)(x)
=\int_{0}^{\infty}\int_{(\mathbb{R}^n)^4}
e^{2\pi ix\cdot(\xi_1+\xi_2+\xi_3+\xi_4)}
m(t\xi_1,t\xi_2)m(t\xi_3,t\xi_4)
\prod_{i=1}^4\hat f_i(\xi_i)
d\xi_i\,\frac{dt}{t}.
\end{align*}

\begin{example}
\label{lem:4-lin-FM-1}
Suppose that $m(0,0)=0$ and there exists some $\varepsilon>0$ such that
\begin{equation}\label{eq:4-lin-FM-1}
|\partial^\alpha m(\xi_1,\xi_2)|
\leq (1+|\xi_1|+|\xi_2|)^{-s-\varepsilon}, \quad \hbox{ \ for \ all \ }|\alpha|\le 2n+1.
\end{equation}

Then, there exists a constant $\delta,$ with $0<\delta\le 1$, such that\begin{enumerate}
\item [(i)]
$\widetilde{\mathcal T}_m$ is bounded from
$L^{q_1}(\mathbb{R}^n)\times L^{q_2}(\mathbb{R}^n)\times L^{q_3}(\mathbb{R}^n)
\times L^{q_4}(\mathbb{R}^n)$
to $L^{q}(\mathbb{R}^n)$ for $2-\delta<q_1,q_2,q_3,q_4<\infty$ with $1/q=1/q_1+1/q_2+1/q_3+1/q_4$.
\item [(ii)]  $\mathfrak{T}_{m}$
is bounded
from $L^{q_1}(\mathbb{R}^n)\times L^{q_2}(\mathbb{R}^n)$
to $L^{q}(\mathbb{R}^n)$ for $2-\delta<q_1,q_2<\infty$ with $1/q=1/q_1+1/q_2$.\end{enumerate}
\end{example}
%
%
%
\begin{proof}
(i) Let $\tilde m(\xi_1,\xi_2,\xi_3,\xi_4)=\int_{0}^{\infty}
m(t\xi_1,t\xi_2) m(t\xi_3,t\xi_4)\frac{dt}{t}.$ Then  $\widetilde{\mathcal T}_m$ can be written as a Fourier multiplier operator in the following form:
\begin{align*}
&\widetilde{\mathcal T}_m(\vec f)(x)
\\
&=\int_{(\mathbb{R}^n)^4}e^{2\pi ix\cdot(\xi_1+\xi_2+\xi_3+\xi_4)} \tilde m(\xi_1,\xi_2,\xi_3,\xi_4)
\hat f_1(\xi_1)\hat f_2(\xi_2)\hat f_3(\xi_3)\hat f_4(\xi_4)
d\xi_1d\xi_2d\xi_3d\xi_4,
\end{align*}

Next, we will show that $\tilde m$ is a multiplier by considering two cases.
\item Case (a): $1\le|\alpha|\le 2n+1$. We have
\begin{align*}
|\partial^\alpha\tilde m(\xi_1,\xi_2,\xi_3,\xi_4)|
&=\bigl|\int_{0}^{\infty}
|\partial^\alpha \bigl(m(t\xi_1,t\xi_2) m(t\xi_3,t\xi_4)\bigr)\frac{dt}{t}
\bigr|
\\
&\leq
\int_{0}^{\infty}\frac{t^{|\alpha|}}
{(1+|t\xi_1|+|t\xi_2|)^{s+\varepsilon}(1+|t\xi_3|+|t\xi_4|)^{s+\varepsilon}}
\frac{dt}{t}
\\
&\le \int_{0}^{\infty}\frac{t^{|\alpha|}}
{(1+t(|\xi_1|+|\xi_2|+|\xi_3|+|\xi_4|))^{s+\varepsilon}}
\frac{dt}{t}
\\
&=\frac{1}{(|\xi_1|+|\xi_2|+|\xi_3|+|\xi_4|)^{|\alpha|}}
 \int_{0}^{\infty}\frac{s^{|\alpha|}}{(1+s)^{s+\varepsilon}}\frac{ds}{s}.
\end{align*}

\item Case (b): $|\alpha|=0$.

 By using the mean-value theorem and the assumption
$m(0,0)=0$, we may obtain that
$|m(\xi_1,\xi_2)|\leq |\xi_1|+|\xi_2|$.
Thus, together with the boundedness of $m$, it yields that
$|m(\xi_1,\xi_2)|\leq (|\xi|+|\xi_2|)^{1/2}$ for $\xi_1,\xi_2\in\mathbb{R}^n$.
Therefore, we have
\begin{align*}
|\tilde m(\xi_1,\xi_2,\xi_3,\xi_4)|
&=\Bigl|\int_{0}^{\infty}
\bigl(m(t\xi_1,t\xi_2) m(t\xi_3,t\xi_4)\bigr)\frac{dt}{t}
\Bigr|
\\
&\leq
\int_{0}^{\infty}\frac{(|t\xi_1|+|t\xi_2|)^{1/4}}
{(1+|t\xi_1|+|t\xi_2|)^{1/2}}\frac{(|t\xi_3|+|t\xi_4|)^{1/4}}
{(1+|t\xi_3|+|t\xi_4|)^{1/2}}
\frac{dt}{t}
\\
&\le \int_{0}^{\infty}\frac{(t(|\xi_1|+|\xi_2|)(t(|\xi_3|+|\xi_4|))^{1/4}}
{(1+t(|\xi_1|+|\xi_2|+|\xi_3|+|\xi_4|))^{1/2}}
\frac{dt}{t}
\\
&\le
 \int_{0}^{\infty}\frac{s^{1/2}}{(1+s)^{1/2}}\frac{ds}{s}<\infty.
\end{align*}
Note that $2n+1>4n/2$, then by Theorem 1 in \cite{si}, one obtains that there exists
$0<\delta\le 1$ such that
$\widetilde{\mathcal T}_m$ is bounded
from $L^{q_1}(\mathbb{R}^n)\times L^{q_2}(\mathbb{R}^n)\times L^{q_3}(\mathbb{R}^n)\times L^{q_4}(\mathbb{R}^n)$
to $L^{q}(\mathbb{R}^n)$ for $2-\delta<q_1,q_2,q_3,q_4$ with
$1/q=1/q_1+1/q_2+1/q_3+1/q_4$.

(ii)
Note that
\begin{align*}
\mathfrak{T}_{m}(\vec f)(x)^2
&=\int_{0}^{\infty}\int_{(\mathbb{R}^n)^4}
e^{2\pi ix\cdot(\xi_1+\xi_2-\xi_3-\xi_4)}
m(t\xi_1,t\xi_2)\overline{m(t\xi_3,t\xi_4)}
\hat f_1(\xi_1)\hat f_2(\xi_2)\\&\quad \times \overline{\hat f_1(\xi_3)\hat f_2(\xi_4)}
d\xi_1d\xi_2d\xi_3d\xi_4\,\frac{dt}{t}
\\
&=\int_{0}^{\infty}\int_{(\mathbb{R}^n)^4}
e^{2\pi ix\cdot(\xi_1+\xi_2+\xi_3+\xi_4)}
m(t\xi_1,t\xi_2)\overline{m(-t\xi_3,-t\xi_4)}
\\&\quad \times \hat f_1(\xi_1)\hat f_2(\xi_2)\overline{\hat f_1(-\xi_3)\hat f_2(-\xi_4)}
d\xi_1d\xi_2d\xi_3d\xi_4\,\frac{dt}{t}.
\end{align*}
Then, as a consequence of (i), we obtain that $\mathfrak{T}_{m}$ is bounded
from $L^{q_1}(\mathbb{R}^n)\times L^{q_2}(\mathbb{R}^n)$
to $L^{q}(\mathbb{R}^n)$ for $2-\delta<q_1,q_2<\infty$ with $1/q=1/q_1+1/q_2$.
\end{proof}

\section{Proofs of Theorems \ref{31} and \ref{34}}\label{Sec-3}

This section will be devoted to prove Theorems \ref{31} and \ref{34} by showing that
the associated kernel of $  \mathfrak{T}_m$ satisfies the conditions (H2) and
(H3) in Definition \ref{1.444}. The following two propositions provide a foundation for our analysis.
\begin{proposition}\label{prop:BuiDuong}
Let $s\in \mathbb{N}$ satisfy $n+1\le s\le 2n$.
Suppose $m\in L^\infty((\mathbb{R}^n)^2)$ satisfies
\begin{equation}\label{eq:2-lin-FM-5-1}
|\partial^\alpha m(\xi_1,\xi_2)|
\lesssim \frac{(|\xi_1|+|\xi_2|)^{-|\alpha|+\varepsilon_1}}
{(1+|\xi_1|+|\xi_2|)^{\varepsilon_1+\varepsilon_2}} ,
\end{equation}
for some $\varepsilon_1, \varepsilon_2>0$ and  $|\alpha|\le s$.
Then, for any $2n/s< p\le 2$, there exist $C>0$ and $\delta>n/p$, such that
\begin{align}
&\biggl(
\int_{S_j(Q)}\int_{S_k(Q)}\Bigl(\int_{0}^\infty \Bigl|
\check m\Bigl(\frac{x-y_1}{t},\frac{x-y_2}{t}\Bigr)-
\check m \Bigl(\frac{\bar{x}-y_1}{t},\frac{\bar{x}-y_2}{t}\Bigr)\Bigr|^2
\frac{dt}{t^{4n+1}}\Bigr)^{\frac{p'}{2}}dy_1dy_2\biggr)^{\frac1{p'}}
\label{eq:FM-diff-1}
\\
&\le C\frac{|x-\bar{x}|^{2(\delta-n/p)}}{|Q|^{2\delta/n}}2^{-2\delta \max(j,k)}
\notag
\end{align}
for all balls $Q$, all $x,\bar{x}\in 1/ 2Q$ and $(j,k)\ne (0,0)$.
\end{proposition}
\begin{proof}
Denote the left-side of \eqref{eq:FM-diff-1} by $A_{j,k}(m,Q)(x, \bar{x})$,
and let $Q=B(x_0,R)$. Let $u=ax$ $(a>0)$ and $s=at$, one obtains
 that
\begin{align*}
A_{j,k}(m,Q)(x,\bar x)
&=a^{-2n/p'}\biggl(\int_{S_j(Q^a)}\int_{S_k(Q^a)}\Bigl(\int_{0}^\infty \Bigl|
\check m\Bigl(\frac{x^a-u_1}{at},\frac{x^a-u_2}{at}\Bigr)\\&\quad
-\check m \Bigl(\frac{\bar{x}^a-u_1}{at},\frac{\bar{x}^a-u_2}{at}\Bigr)
\Bigr|^2\frac{dt}{t^{4n+1}}\Bigr)^{\frac{p'}{2}}du_1du_2
\biggr)^{\frac1{p'}}
\\
&=a^{2n/p}\biggl(\int_{S_j(Q^a)}\int_{S_k(Q^a)}\Bigl(\int_{0}^\infty
\Bigl|\check m\Bigl(\frac{x^a-u_1}{s},\frac{x^a-u_2}{s}\Bigr)\\&\quad
-\check m \Bigl(\frac{\bar{x}^a-u_1}{s},\frac{\bar{x}^a-u_2}{s}\Bigr)
\Bigr|^2\frac{ds}{s^{4n+1}}\Bigr)^{\frac{p'}{2}}du_1du_2
\biggr)^{\frac1{p'}}\\
&=a^{2n/p} A_{j,k}(m,Q^a)(x^a, \bar{x}^a),
\end{align*}
where $Q^a=B(ax_0,aR)$, $x^{a}=ax$ and $\bar x^{a}=a \bar{x}$.
Therefore, taking $a=1/(2^{\max(j,k)}R)$, the desired estimate \eqref{eq:FM-diff-1} follows from the following fact:
\begin{equation*}
A_{j,k}(m,Q^a)(x^a, \bar{x}^a)\lesssim
\frac{|x^a-\bar{x}^a|^{2(\delta-n/p)}}{|Q^a|^{2\delta/n}}2^{-2\delta \max(j,k)}
=|x^a-\bar{x}^a|^{2(\delta-n/p)}
\end{equation*}
Thus, we only need to show
\eqref{eq:FM-diff-1} in the case $R=1/2^{\max(j,k)}$. In addition, we may assume
$|h|=|x-\bar{x}|<1/2$ and $k\ge j$ (hence $k\ge1$).
Hence, for $Q=B(x_0,2^{-k})$ and $\delta>n/p$, we need to show that
\begin{equation}\label{eq:BuiDuong-reduced}
A_{j,k}(m,Q)(x,\bar x)\lesssim  |x-\bar{x}|^{2(\delta-n/p)}.
\end{equation}


Let $\Psi\in \mathcal{S}(\mathbb{R}^{2n})$ with
$\operatorname{supp}\Psi \in \{ (\xi,\eta): 1/ 2\leq |\xi|+|\eta|\leq 2\}$ and
$$
\sum_{j\in \mathbb{Z}}\Psi(2^{-j}\xi,2^{-j}\eta)=1, \quad\quad\hbox{\ for \ all } (\xi,\eta)\in (\mathbb{R}^{2n})\setminus \{0 \}.
$$
\par Now, we can write
\begin{eqnarray*}\label{decomposition}
 m(\xi,\eta)=\sum_{j\in \mathbb{Z}}m_j(\xi,\eta)
 :=\sum_{j\in \mathbb{Z}}\Psi(2^{-j}\xi,2^{-j}\eta)m(\xi,\eta)
\end{eqnarray*}
and hence $\operatorname{supp}m_j\subseteq
\{ (\xi,\eta):  2^{j-1}\leq |\xi|+|\eta|\leq 2^{j+1}\}$.

By changing variables, to prove \eqref{eq:BuiDuong-reduced}, it is sufficient to show that for $Q=B(x_0,2^{-k})$, the following inequality holds:
\begin{align*}
&\biggl(
\int_{S_j(Q_{\bar{x}})}\int_{S_k(Q_{\bar{x}})}\Bigl(\int_0^\infty \Bigl|
\check m\Bigl(\frac{y+h}{t},\frac{z+h}{t}\Bigr)
-\check m \Bigl(\frac{y}{t},\frac{z}{t}\Bigr) \Bigr|^2
\frac{dt}{t^{4n+1}}\Bigr)^{\frac{p'}{2}}dydz\biggr)^{\frac1{p'}}
\le C{|h|^{2(\delta-n/p)}},
\end{align*}
where $h=x-\bar{x}$ and $Q_{\bar{x}}=Q-\bar{x}$.
We prove this in the following three cases.

\textbf{(a) The case $2n/p<s<2n/p+1$}.
First, we note that \eqref{eq:2-lin-FM-5-1} remains valid for any smaller
positive number than $\varepsilon_1$. Thus, one may take $\varepsilon_1$ sufficiently
close to $s-2n/p$ so that $0<\varepsilon_1<s-2n/p$.

For any interval $I$ in $\mathbb R_+$, we introduce the notion $A_\ell$ and $A_\ell(I)$ as follows.
\begin{align*}
A_\ell&:=\biggl(
\int_{S_j(Q_{\bar{x}})}\int_{S_k(Q_{\bar{x}})}\Bigl(\int_0^\infty\Bigl|
\check m_\ell\Bigl(\frac{y+h}{t},\frac{z+h}{t}\Bigr)-
\check m_\ell\Bigl(\frac{y}{t},\frac{z}{t}\Bigr) \Bigr|^2
\frac{dt}{t^{4n+1}}\Bigr)^{\frac{p'}{2}}dydz\biggr)^{\frac1{p'}};
\\
A_\ell(I)&:=\biggl(
\int_{S_j(Q_{\bar{x}})}\int_{S_k(Q_{\bar{x}})}\Bigl(\int_I \Bigl|
\check m_\ell\Bigl(\frac{y+h}{t},\frac{z+h}{t}\Bigr)-
\check m_\ell\Bigl(\frac{y}{t},\frac{z}{t}\Bigr) \Bigr|^2
\frac{dt}{t^{4n+1}}\Bigr)^{\frac{p'}{2}}dydz\biggr)^{\frac1{p'}}.
\end{align*}

Since $Q_{\bar x}=B(x_0-\bar x, 1/2^k)$, we have
$2^{-2}\le |y+h|\le 2$ and $|z+h|\le 2^{j-k+1}$ for all $y\in S_k(Q_{\bar x})$
and $z\in S_j(Q_{\bar x})$. Therefore, it yields that
\begin{equation*}
A_\ell(I)\lesssim
\biggl(\int_{|z|\le 2^{j-k+1}}\int_{2^{-2}\le |y|\le2}\Bigl(\int_I \Bigl|
\check m_\ell\Bigl(\frac{y}{t},\frac{z}{t}\Bigr)\Bigr|^2
\frac{dt}{t^{4n+1}}\Bigr)^{\frac{p'}{2}}dydz\biggr)^{\frac1{p'}}.
\end{equation*}
Note that $|y|\sim 1$ in the above integration domain, by the Minkowski inequality
and the Haussdorf-Young inequality, for $|\alpha|=s$, we have
\begin{align}
A_\ell(I)&\lesssim
\biggl(\int_{|z|\le 2^{j-k+1}}\int_{2^{-2}\le |y|\le2}\Bigl(\int_I |y^\alpha|
\Bigl|\check m_\ell\Bigl(\frac{y}{t},\frac{z}{t}\Bigr)\Bigr|^2
\frac{dt}{t^{4n+1}}\Bigr)^{\frac{p'}{2}}dydz\biggr)^{\frac1{p'}} \notag
\\
&\le \biggl(\int_I \biggl(
\int_{|z|\le 2^{j-k+1}}\int_{2^{-2}\le |y|\le2}\Bigl|y^{\alpha}\check m_\ell
\Bigl(\frac{y}{t},\frac{z}{t}\Bigr)\Bigr|^{p'}dydz
\biggr)^{\frac2{p'}}\frac{dt}{t^{4n+1}}\biggr)^{\frac1{2}} \notag
\\
&= \biggl(\int_I \biggl(
\int_{|tz|\le 2^{j-k+1}}\int_{2^{-2}\le |ty|\le2}|y^{\alpha}\check m_\ell
(y,z)|^{p'}dydz
\biggr)^{\frac2{p'}}t^{2|\alpha|+4n/p'}\frac{dt}{t^{4n+1}}\biggr)^{\frac1{2}}
\notag
\\
&\le \biggl(\int_I \biggl(
\int_{\mathbb{R}^n}\int_{\mathbb{R}^n}
|\partial_\xi^{\alpha} m_\ell(\xi,\eta)|^{p}d\xi d\eta
\biggr)^{\frac2{p}}t^{2|\alpha|-4n/p-1}dt\biggr)^{\frac1{2}} \notag
\\
&\lesssim
\biggl(\int_I t^{2|\alpha|-4n/p-1}dt\biggr)^{\frac1{2}}\biggl(
\int_{\mathbb{R}^n}\int_{\mathbb{R}^n}
|\partial_\xi^{\alpha} m_\ell(\xi,\eta)|^{p}d\xi d\eta
\biggr)^{\frac1{p}}. \notag
\end{align}
Hence, we obtain
\begin{equation}\label{eq:FM-diff-2}
A_\ell(I)\lesssim
\frac{(2^{\ell})^{\varepsilon_1-|\alpha|+2n/p}}
{(1+2^\ell)^{\varepsilon_1+\varepsilon_2}}
\biggl(\int_I t^{2|\alpha|-4n/p-1}dt\biggr)^{\frac1{2}}.
\end{equation}
Now, setting $\varphi_{\ell}(\xi,\eta)
=m_{\ell}(\xi,\eta)(e^{2\pi it^{-1}h\cdot(\xi+\eta)}-1)$, we have
\begin{equation*}
\check m_\ell\Bigl(\frac{y+h}{t},\frac{z+h}{t}\Bigr)
-\check m_\ell\Bigl(\frac{y}{t},\frac{z}{t}\Bigr)
=\check \varphi_{\ell}\Bigl(\frac{y}{t},\frac{z}{t}\Bigr).
\end{equation*}
Proceeding the same argument as before, we have
\begin{align*}
A_\ell(I)&\lesssim \biggl(
\int_{S_j(Q_{\bar{x}})}\int_{S_k(Q_{\bar{x}})}\Bigl(
\int_I \Bigl|y^{\alpha}\Bigl(
\check m_\ell\Bigl(\frac{y+h}{t},\frac{z+h}{t}\Bigr)-
\check m_\ell\Bigl(\frac{y}{t},\frac{z}{t}\Bigr)\Bigr)\Bigr|^2
\frac{dt}{t^{4n+1}}\Bigr)^{\frac{p'}{2}}dydz\biggr)^{\frac1{p'}}
\\
&= \biggl(\int_I \biggl(\int_{S_j(Q_{\bar{x}})}\int_{S_k(Q_{\bar{x}})}
\Bigl|y^{\alpha}\check \varphi_\ell
\Bigl(\frac{y}{t},\frac{z}{t}\Bigr)\Bigr|^{p'}dydz
\biggr)^{\frac2{p'}}\frac{dt}{t^{4n+1}}\biggr)^{\frac1{2}}
\\
&= \biggl(\int_I \biggl(
\int_{S_j(t^{-1}Q_{\bar{x}})}\int_{S_k(t^{-1}Q_{\bar{x}})}
|y^{\alpha}\check \varphi_\ell(y,z)|^{p'}dydz
\biggr)^{\frac2{p'}}t^{2|\alpha|+4n/p'}\frac{dt}{t^{4n+1}}\biggr)^{\frac1{2}}
\\
&\le \biggl(\int_I \biggl(
\int_{\mathbb{R}^n}\int_{\mathbb{R}^n}
|\partial_\xi^{\alpha} \varphi_\ell(\xi,\eta)|^{p}d\xi d\eta
\biggr)^{\frac2{p}}t^{2|\alpha|-4n/p-1}dt\biggr)^{\frac1{2}}
\\
&= \biggl(\int_I \biggl(
\int_{\mathbb{R}^n}\int_{\mathbb{R}^n}|\partial_\xi^{\alpha}
[m_\ell(\xi,\eta)(e^{-2\pi it^{-1}h\cdot(\xi+\eta)}-1)]|^{p}d\xi d\eta
\biggr)^{\frac2{p}}t^{2|\alpha|-4n/p-1}dt\biggr)^{\frac1{2}}.
\end{align*}
By the following fact
\begin{equation*}
|\partial_\xi^{\alpha}[m_\ell(\xi,\eta)(e^{-2\pi it^{-1}h\cdot(\xi+\eta)}-1)]|
\lesssim \frac{2^{\ell}|h|}{t}\frac{(2^{\ell})^{\varepsilon_1-|\alpha|}}
{(1+2^\ell)^{\varepsilon_1+\varepsilon_2}}
+
\sum_{\beta=1}^{|\alpha|}\Bigl(\frac{|h|}{t}\Bigr)^\beta
\frac{(2^{\ell})^{\varepsilon_1-|\alpha|+\beta}}
{(1+2^\ell)^{\varepsilon_1+\varepsilon_2}},
\end{equation*}
it yields that
\begin{align}
A_\ell(I)&\lesssim
\biggl(\int_I\Bigl(\frac{2^{\ell}|h|}{t}
\frac{(2^{\ell})^{\varepsilon_1-|\alpha|}}
{(1+2^\ell)^{\varepsilon_1+\varepsilon_2}}
+
\sum_{\beta=1}^{|\alpha|}\Bigl(\frac{|h|}{t}\Bigr)^\beta
\frac{(2^{\ell})^{\varepsilon_1-|\alpha|+\beta}}
{(1+2^\ell)^{\varepsilon_1+\varepsilon_2}}\Bigr)^2
2^{4n\ell/p}t^{2|\alpha|-4n/p-1}dt\biggr)^{\frac1{2}}\label{eq:FM-diff-3}
\\
&\lesssim
\sum_{\beta=0}^{|\alpha|}|h|^{\max(\beta,1)}\frac
{2^{\ell(-|\alpha|+2n/p+\max(\beta,1)+\varepsilon_1)}}
{(1+2^\ell)^{\varepsilon_1+\varepsilon_2}}
\biggl(\int_I t^{2(|\alpha|-2n/p-\max(\beta,1))-1}dt\biggr)^{\frac1{2}}.\notag
\end{align}
Now, we fix sufficiently small $\varepsilon>0$ so that
$\varepsilon(s-2n/p)<\min\{\varepsilon_1, \varepsilon_2\}$.
Then, if $2^\ell |h|\ge 1$, noting $2n/p<s<2n/p+1$
and using \eqref{eq:FM-diff-2} for $I=(0, (2^\ell |h|)^{1+\varepsilon}]$,
we have
\begin{equation*}
A_\ell((0, (2^\ell |h|)^{1+\varepsilon}])\lesssim
2^{-\ell(s+\varepsilon_2-2n/p)}(2^\ell |h|)^{(1+\varepsilon)(s-2n/p)}
=|h|^{(1+\varepsilon)(s-2n/p)}2^{\ell(\varepsilon(s-2n/p)-\varepsilon_2)}.
\end{equation*}
By \eqref{eq:FM-diff-3} for $I=[(2^\ell |h|)^{1+\varepsilon}, \infty)$,
we have
\begin{align*}
A_\ell([(2^\ell |h|)^{1+\varepsilon}, \infty))
&\lesssim
\sum_{\beta=0}^{|\alpha|}|h|^{\max(\beta,1)}
{2^{\ell(-|\alpha|+2n/p+\max(\beta,1))}}
(2^\ell |h|)^{(1+\varepsilon)(s-2n/p-\max(\beta,1))}
\\
&=\sum_{\beta=0}^{|\alpha|}
|h|^{-\varepsilon\max(\beta,1) +(1+\varepsilon)(s-2n/p)}
2^{\ell\varepsilon((s-2n/p)-\max(\beta,1))}.
\end{align*}
Thus, noting $\varepsilon(s-2n/p)-\varepsilon_2<0$ and $|h|<1$, we obtain
\begin{align}
\sum_{2^\ell |h|\ge 1} A_\ell
&\lesssim \sum_{2^\ell |h|\ge 1}
|h|^{(1+\varepsilon)(s-2n/p)}2^{\ell(\varepsilon(s-2n/p)-\varepsilon_2)}
\label{eq:FM-diff-4}
\\
&\ +\sum_{2^\ell |h|\ge 1}\sum_{\beta=0}^{|\alpha|}
|h|^{-\varepsilon\max(\beta,1) +(1+\varepsilon)(s-2n/p)}
2^{\ell\varepsilon((s-2n/p)-\max(\beta,1))} \notag
\\
&\le |h|^{s-2n/p+\varepsilon_2}+\sum_{\beta=0}^{|\alpha|}|h|^{s-2n/p}
 \lesssim |h|^{s-2n/p}. \notag
\end{align}
In the case $2^\ell|h|<1$,
using \eqref{eq:FM-diff-2} for $I=(0, (2^\ell |h|)^{1-\varepsilon}]$,
we have
\begin{equation*}
A_\ell((0, (2^\ell |h|)^{1-\varepsilon}])\lesssim
2^{\ell(-s+2n/p+\varepsilon_1)}(2^\ell |h|)^{(1-\varepsilon)(s-2n/p)}
=|h|^{(1-\varepsilon)(s-2n/p)}2^{\ell(-\varepsilon(s-2n/p)+\varepsilon_1)}.
\end{equation*}
Further more, by using \eqref{eq:FM-diff-3} for $I=[(2^\ell |h|)^{1-\varepsilon}, \infty)$,
we have
\begin{align*}
A_\ell([(2^\ell |h|)^{1-\varepsilon}, \infty))
&\lesssim
\sum_{\beta=0}^{|\alpha|}|h|^{\max(\beta,1)}
{2^{\ell(-s+2n/p+\max(\beta,1))}}
(2^\ell |h|)^{(1-\varepsilon)(s-2n/p-\max(\beta,1))}
\\
&=\sum_{\beta=0}^{|\alpha|}
|h|^{\varepsilon\max(\beta,1)+(1-\varepsilon)(s-2n/p)}
2^{-\varepsilon\ell(s-2n/p-\max(\beta,1))}.
\end{align*}
By the fact that $\varepsilon(s-2n/p)-\varepsilon_1<0$ and $|h|<1$, we obtain
\begin{align}
\sum_{2^\ell |h|< 1} A_\ell
&\lesssim \sum_{2^\ell |h|< 1}
|h|^{(1-\varepsilon)(s-2n/p)}2^{\ell(-\varepsilon(s-2n/p)+\varepsilon_1)}
\label{eq:FM-diff-5}
\\
&\ +\sum_{2^\ell |h|< 1}\sum_{\beta=0}^{|\alpha|}
|h|^{\varepsilon\max(\beta,1)+(1-\varepsilon)(s-2n/p)}
2^{-\varepsilon\ell(s-2n/p-\max(\beta,1))} \notag
\\
&\le |h|^{s-2n/p-\varepsilon_1}+\sum_{\beta=0}^{|\alpha|}|h|^{s-2n/p}
 \lesssim |h|^{s-2n/p-\varepsilon_1}+|h|^{s-2n/p}. \notag
\end{align}
Noting that $0<\varepsilon_1<s-2n/p$ and taking $\delta=(s-\varepsilon_1)/2$, by \eqref{eq:FM-diff-4} and
\eqref{eq:FM-diff-5}, it holds that
\begin{align*}
&\biggl(
\int_{S_j(Q_{\bar{x}})}\int_{S_k(Q_{\bar{x}})}\Bigl(\int_0^\infty \Bigl|
\check m\Bigl(\frac{y+h}{t},\frac{z+h}{t}\Bigr)
-\check m \Bigl(\frac{y}{t},\frac{z}{t}\Bigr) \Bigr|^2
\frac{dt}{t^{4n+1}}\Bigr)^{\frac{p'}{2}}dydz\biggr)^{\frac1{p'}}
\\
&\le \sum_{\ell\in\mathbb Z}A_\ell
\lesssim{|h|^{2(\delta-n/p)}},
\end{align*}
This leads to the conclusion of Proposition \ref{prop:BuiDuong}
in the case $2n/p<s<2n/p+1$.

\textbf{(b) The case $2n/p<s=2n/p+1$}. First, we Choose $1<p_0<p$ such that $2n/p_0<s$. Then $p_0$ satisfies
$2n/p_0<s=2n/p+1<2n/p_0+1$. Hence, for all balls $Q$, all $x,\bar x\in \frac12 Q$ and $(j,k)\ne (0,0)$, by the step (a), we have
\begin{align*}
&\biggl(
\int_{S_j(Q_{\bar{x}})}\int_{S_k(Q_{\bar{x}})}\Bigl(\int_I \Bigl|
\check m\Bigl(\frac{y+h}{t},\frac{z+h}{t}\Bigr)-
\check m\Bigl(\frac{y}{t},\frac{z}{t}\Bigr) \Bigr|^2
\frac{dt}{t^{4n+1}}\Bigr)^{\frac{p'_0}{2}}dydz\biggr)^{\frac1{p'_0}}
\\
&\le C\frac{|h|^{2\delta-2n/p_0}}{|Q|^{2\delta/n}}2^{-2\delta \max(j,k)}.
\end{align*}
By the H\"older inequality, it yields that
\begin{align*}
&\biggl(
\int_{S_j(Q_{\bar{x}})}\int_{S_k(Q_{\bar{x}})}\Bigl(\int_I \Bigl|
\check m\Bigl(\frac{y+h}{t},\frac{z+h}{t}\Bigr)-
\check m\Bigl(\frac{y}{t},\frac{z}{t}\Bigr) \Bigr|^2
\frac{dt}{t^{4n+1}}\Bigr)^{\frac{p'}{2}}dydz\biggr)^{\frac1{p'}}
\\
&\le (2^{n(j+k)}|Q|^2)^{\frac{1}{p_0}-\frac{1}{p}}\biggl(
\int_{S_j(Q_{\bar{x}})}\int_{S_k(Q_{\bar{x}})}\Bigl(\int_I \Bigl|
\check m\Bigl(\frac{y+h}{t},\frac{z+h}{t}\Bigr)-
\check m\Bigl(\frac{y}{t},\frac{z}{t}\Bigr) \Bigr|^2\\&\quad \times
\frac{dt}{t^{4n+1}}\Bigr)^{\frac{p_0'}{2}}dydz\biggr)^{\frac1{p_0'}}
\\
&\lesssim(2^{2n\max(j,k)}|Q|^2)^{\frac{1}{p_0}-\frac{1}{p}}
\frac{|h|^{2\delta-2n/p_0}}{|Q|^{\frac {2\delta}n}}
\frac{1}{2^{2\delta\max(j,k)}}
\\
&=\frac{|h|^{({2\delta}-2n/p_0+2n/p)-2n/p}}
{|Q|^{\frac{({2\delta}-2n/p_0+2n/p)}{n}}}
{2^{-({2\delta}-2n/p_0+2n/p)\max(j,k)}}.
\end{align*}
Therefore, taking $\delta-n/p_0+n/p>n/p$ as $\delta$ newly,
we obtain the desired estimate.

\textbf{(c) The case $2n/p+1<s\le 2n$}.
In this case there is an integer $l$ such that $2n/p+l<s\le 2n/p+1+l$. Then
it follows that $2n/p<s-l\le 2n/p+1$. Thus, regarding $s-l$ as $s$,
we may deduce this case to the previous case (a) or case (b). This completes the proof of
Proposition  \ref{prop:BuiDuong}.

\end{proof}
\begin{proposition}\label{prop:BuiDuong2}
Let $s\in \mathbb{N}$ with $n+1\le s\le 2n$.
Let $m\in L^\infty((\mathbb{R}^n)^2)$ and satisfy
\begin{equation}\label{eq:2-lin-FM-5-2}
|\partial^\alpha m(\xi_1,\xi_2)|
\lesssim (|\xi_1|+|\xi_2|)^{-|\alpha|} \text{ for }|\alpha|\le s,
\end{equation}
and
\begin{equation}\label{eq:2-lin-FM-5-3}
|m(\xi_1,\xi_2)|
\lesssim\frac{(|\xi_1|+|\xi_2|)^{\varepsilon_1}}
{(1+|\xi_1|+|\xi_2|)^{\varepsilon_1+\varepsilon_2}}, \quad \hbox{\ for \ some } \varepsilon_1,\varepsilon_2>0.
\end{equation}
Then, for $2n/s<p\le 2$, there exists a constant $C>0$, such that the following inequality holds for all balls $Q$ with center at $x $ and $(j,k)\ne (0,0)$.
\begin{equation}\label{eq:FM-ker-1}
\bigl(
\int_{S_j(Q )}\int_{S_k(Q )}\Bigl(\int_0^\infty \Bigl|
\check m \Bigl(\frac{x-y_1}{t},\frac{x-y_2}{t}\Bigr)\Bigr|^2
\frac{dt}{t^{4n+1}}\Bigr)^{\frac{p'}{2}}dy_1dy_2\bigr)^{\frac1{p'}}
\le C\frac{1}{|Q|^{2/p}}2^{-2n\max(j,k)/p}.
\end{equation}

\end{proposition}
\begin{proof}
Let $Q=B(x,R)$, $u=ax$ $(a>0)$ and $s=at$, we have
\begin{align*}
B_{j,k}(m,Q)(x)&:=\bigl(
\int_{S_j(Q )}\int_{S_k(Q )}\Bigl(\int_0^\infty \Bigl|
\check m \Bigl(\frac{x-y_1}{t},\frac{x-y_2}{t}\Bigr)\Bigr|^2
\frac{dt}{t^{4n+1}}\Bigr)^{\frac{p'}{2}}dy_1dy_2\bigr)^{\frac1{p'}}
\\
&=a^{2n/p}\biggl(\int_{S_j(Q^a )}\int_{S_k(Q^a )}
\Bigl(\int_{0}^\infty\Bigl|\check m\Bigl(\frac{x^a-u_1}{t},\frac{x^a-u_2}{t}\Bigr)\Bigr|^2\frac{ds}{s^{4n+1}}\Bigr)^{\frac{p'}{2}}du_1du_2
\biggr)^{\frac1{p'}}
\\
&=a^{2n/p}B_{j,k}(m,Q^a)(x^a),
\end{align*}
where $Q^a=B(ax,aR)$, $x^{a}=ax$.
So, taking $a=1/(2^{\max(j,k)}R)$, the estimate
$B_{j,k}(m,Q^a)(x^a)\lesssim  1$
implies the desired estimate. Thus, we only need to show
\eqref{eq:FM-ker-1} in the case $R=1/2^{\max(j,k)}$.
We may also assume $k\ge j$ and hence $k\ge1$.
Then, for $Q=B(x,2^{-k})$, it sufficient to show that
\begin{equation*}\label{eq:BuiDuong2-reduced}
B_{j,k}(m,Q)(x)\lesssim  1.
\end{equation*}
By changing variables,  it is enough to show that

\begin{equation}\label{eq:FM-ker-11}
\bigl(
\int_{S_j(Q_{x})}\int_{S_k(Q_{x})}\Bigl(\int_0^\infty \Bigl|
\check m \Bigl(\frac{y}{t},\frac{z}{t}\Bigr)\Bigr|^2
\frac{dt}{t^{4n+1}}\Bigr)^{\frac{p'}{2}}dydz\bigr)^{\frac1{p'}}
\le C\frac{1}{|Q|^{2/p}}2^{-2n\max(j,k)/p},
\end{equation}
where  $Q_{x}=Q-x$.

For every interval $I$ in $\mathbb R_+$, let
\begin{equation*}\label{eq:BuiDuong2-3}
B_{j,k}(m_\ell,Q,I)(x):=\biggl(
\int_{S_j(Q_x)}\int_{S_k(Q_x)}\biggl(\int_I \Bigl|
\check m_\ell \Bigl(\frac{y}{t},\frac{z}{t}\Bigr)\Bigr|^2
\frac{dt}{t^{4n+1}}\Bigr)^{\frac{p'}{2}}dydz\biggr)^{\frac1{p'}}.
\end{equation*}
The Minkowski inequality, together with Haussdorf-Young inequality implies that


\begin{align*}\label{eq:BuiDuong2-4}
&B_{j,k}(m_\ell,Q,I)(x)
\\&\lesssim(2^kR)^{-|\alpha|}\bigl(
\int_{S_j(Q_{x})}\int_{S_k(Q_{x})}\Bigl(\int_I \Bigl|y^{\alpha}\check m_\ell
\Bigl(\frac{y}{t},\frac{z}{t}\bigr)\Bigr|^2\frac{dt}{t^{4n+1}}
\Bigr)^{\frac{p'}{2}}dydz\bigr)^{\frac1{p'}}
\\
&\lesssim(2^kR)^{-|\alpha|}\bigl(\int_I \bigl(
\int_{S_j(Q_{x})}\int_{S_k(Q_{x})}\Bigl|y^{\alpha}\check m_\ell
\Bigl(\frac{y}{t},\frac{z}{t}\Bigr)\Bigr|^{p'}dydz
\bigr)^{\frac2{p'}}\frac{dt}{t^{4n+1}}\bigr)^{\frac1{2}}
\\
&= C(2^kR)^{-|\alpha|}\bigl(\int_I \bigl(
\int_{S_j(t^{-1}Q_{x})}\int_{S_k(t^{-1}Q_{x})}|y^{\alpha}\check m_\ell
(y,z)|^{p'}dydz
\bigr)^{\frac2{p'}}t^{2|\alpha|+4n/p'}\frac{dt}{t^{4n+1}}\bigr)^{\frac1{2}}
\\
&\lesssim(2^kR)^{-|\alpha|}\bigl(\int_I \bigl(
\int_{\mathbb{R}^n}\int_{\mathbb{R}^n}|\partial_\xi^{\alpha} m_\ell(\xi,\eta)|^{p}d\xi d\eta
\bigr)^{\frac2{p}}t^{2|\alpha|-4n/p-1}dt\bigr)^{\frac1{2}}
\\
&\lesssim (2^kR)^{-|\alpha|}
\bigl(\int_I t^{2|\alpha|-4n/p-1}dt\bigr)^{\frac1{2}}\bigl(
\int_{\mathbb{R}^n}\int_{\mathbb{R}^n}|\partial_\xi^{\alpha} m_\ell(\xi,\eta)|^{p}d\xi d\eta
\bigr)^{\frac1{p}}.
\end{align*}
Next, we consider two cases according to the value of $\ell.$
\par
\textbf{\item Case (a).  $\ell<0$}.   In this case,  taking $|\alpha|=0$ and $I=[2^{\ell(1+\varepsilon)}, \infty)$, the estimate in \eqref{eq:2-lin-FM-5-3} implies that
\begin{equation*}
B_{j,k}(m_\ell,Q,[2^{\ell(1+\varepsilon)}, \infty))
\lesssim
2^{\ell(1+\varepsilon)(-2n/p)}2^{\ell\varepsilon_1}2^{\ell(2n/p)}
=2^{\ell(\varepsilon_1-2\varepsilon n/p)}.
\end{equation*}
In virtue of $2^kR=1$, taking $|\alpha|=s$ and
$I=[0, 2^{\ell(1+\varepsilon)}]$, the estimate in \eqref{eq:2-lin-FM-5-2} implies that

\begin{equation*}
B_{j,k}(m_\ell,Q,[0, 2^{\ell(1+\varepsilon)}])
\lesssim
2^{\ell(1+\varepsilon)(s-2n/p)}2^{-\ell(s-2n/p)}
=2^{\ell\varepsilon(s-2n/p)}.
\end{equation*}
Hence,
\begin{equation*}
B_{j,k}(m_\ell,Q,[0, \infty))\lesssim
2^{\ell(\varepsilon_1-2\varepsilon n/p)}+2^{\ell\varepsilon(s-2n/p)}.
\end{equation*}

\par
\textbf{\item Case (b).  $\ell\ge 0$}.   By repeating the same arguments as in case (a), we get
\begin{equation*}
B_{j,k}(m_\ell,Q,[2^{\ell(1-\varepsilon)}, \infty))
\lesssim
2^{\ell(1-\varepsilon)(-2n/p)}2^{-\ell\varepsilon_2}2^{\ell(2n/p)}
=2^{\ell(2\varepsilon n/p-\varepsilon_2)}
\end{equation*}
and
\begin{equation*}
B_{j,k}(m_\ell,Q,[0, 2^{\ell(1+\varepsilon)}])
\lesssim
2^{\ell(1-\varepsilon)(s-2n/p)}2^{-\ell(s-2n/p)}
=2^{-\ell\varepsilon(s-2n/p)}.
\end{equation*}
Therefore,
\begin{equation*}
B_{j,k}(m_\ell,Q,[0, \infty))\lesssim
2^{\ell(2\varepsilon n/p-\varepsilon_2)}+2^{-\ell\varepsilon(s-2n/p)}.
\end{equation*}

Choosing $\varepsilon>0$ so that
$2n\varepsilon/p<\min(\varepsilon_1, \varepsilon_2)$,  we obtain from case (a) and case (b)
\begin{align*}
B_{j,k}(m,Q)(x)&\le \sum_{\ell<0}B_{j,k}(m_\ell,Q,[0, \infty))
+ \sum_{\ell\ge 0}B_{j,k}(m_\ell,Q,[0, \infty))
\\
&\lesssim \sum_{\ell<0}
[2^{\ell(\varepsilon_1-2\varepsilon n/p)}+2^{\ell\varepsilon(s-2n/p)}]
+ \sum_{\ell\ge 0}
[2^{\ell(2\varepsilon n/p-\varepsilon_2)}+2^{-\ell\varepsilon(s-2n/p)}]\\
&\lesssim 1.
\end{align*}
This completes the proof of Proposition \ref{prop:BuiDuong2}.
\end{proof}
\textbf{Proofs of Theorems \ref{31} and \ref{34}.}
First, we assume that Theorems \ref{1}--\ref{3} are true,
whose proofs will be postponed to the next sections.\\
(a) The case $p_0>2n/s$.
By Proposition \ref{prop:BuiDuong} and Proposition \ref{prop:BuiDuong2}, it is easy to see
 that the associated kernel of $  \mathfrak{T}_m$ satisfies the conditions
 (H2) and (H3).
Since we have supposed (H1) from the beginning,
applying Theorems \ref{1}--\ref{3},
we obtain the desired conclusions in Theorem \ref{31} and  Theorem \ref{34}.\\
(b) The case $p_0=2n/s$. By the property of $A_p$ weights,
there exists a real number $\tilde p_0$ satisfying $p_0=2n/s<\tilde p_0<\min(p_1,p_2,2)$ and
$\vec\omega\in A_{\vec p/{\tilde p_0}}$ (see \cite{anh} or \cite{lerner}).
Therefore, by step (a), we finish the proof of Theorem \ref{31} and Theorem \ref{34}.
\section{Proofs of Theorems \ref{1} and \ref{2}}\label{SEC4}
Let us recall the definition of $A_{\vec{P}}$ weights introduced by Lerner et al.\cite{lerner}.
\begin{definition}\label{def3}
Let $\vec{P}=(p_1,\cdots,p_m)$ and $1/p=1/p_1+\cdots+1/p_m$ with $1\leq p_1,\cdots,p_m<\infty.$  Given
$\vec{\omega}=(\omega_1,\cdots, \omega_m)$, set $\nu_{\vec{\omega}}=\prod_{i=1}^m\omega_i^{p/ p_i}.$
We say that $\vec{\omega}$ satisfies the $A_{\vec{P}}$ condition if
$$\sup_Q \big(\frac{1}{|Q|}\int_Q \prod_{i=1}^m\omega_i^{\frac{p}{p_i}}
\big)^{\frac{1}{p}}\prod_{i=1}^m\big( \frac{1}{|Q|}\int_Q
\omega_i^{1-p_i'} \big)^{\frac{1}{p_i'}}< \infty,$$ when
$p_i=1,$ $\big( \frac{1}{|Q|}\int_Q
\omega_i^{1-p_i'}\big)^{\frac{1}{p_i'}}$ is understood as
$(\inf_Q \omega_i)^{-1}.$
\end{definition}

The new maximal function $\mathcal{M}_p$ can be defined by
$$\mathcal{M}_p(\vec{f})(x)=\sup_{Q\ni x} \prod_{j=1}^m \big( \frac{1}{|Q|}\int_Q |f_j(y_j)|^pdy_j\big)^{1/p}.$$
When $p=1$, we get $\mathcal{M}_1=\mathcal{M}$, which was introduced in \cite{lerner}.

In order to prove our results, we need the following lemmas.
\begin{lemma}\label{22}
 For any $1<p_0\leq p_1,\dots,p_m<\infty$ and $p$ so that $1/p=1/p_1+\cdots +1/p_m$ and $\vec{\omega}\in A_{\vec{P}/p_0},$ where $\vec{P}/p_0=(p_1/p_0,\dots, p_m/p_0)$, the following weighted estimates hold.
 \begin{enumerate}
   \item If there is no $p_i=p_0$, then $\|\mathcal{M}_{p_0}(\vec{f}) \|_{L^{p}(\nu_{\vec{w}})}\leq C \prod_{i=1}^{m} \|f_i\|_{L^{p_i}(w_i)} .$\\
   \item If there is a $p_i=p_0$, then $\|\mathcal{M}_{p_0}(\vec{f}) \|_{L^{p,\infty}(\nu_{\vec{w}})}\leq C \prod_{i=1}^{m} \|f_i\|_{L^{p_i}(w_i)}. $
 \end{enumerate}
\end{lemma}
\begin{proof}
The proof of (1) was given in \cite{bui}. The proof of (2) is similar to (1), we omit the proof.
\end{proof}
\begin{lemma}\label{21}(\cite{fefferman})
 Let $\omega$ be an $A_\infty$ weight. Then there exist constant C and $\rho>0$ depending upon the $A_\infty$ condition of
$\omega$ such that, for all $\lambda,\varepsilon>0,$
$$\omega(\{y\in \mathbb{R}^n: M f(y)>\lambda, M^\sharp f(y)\leq \lambda \epsilon\})\leq C \varepsilon^\rho\omega(\{ y\in \mathbb{R}^n: Mf(y)>\frac 12 \lambda\}).$$
As consequences, we have the following estimates for $\delta>0.$
\par (1) Let $\varphi: (0,\infty)\rightarrow (0,\infty)$ be a doubling, that is, $\varphi(2a)\leq C \varphi(a)$ for $a>0$. Then, there exists a constant C depending upon the $A_\infty$ condition of $\omega$ and doubling condition of $\varphi$ such that
\begin{equation}
\sup_{\lambda>0}\varphi(\lambda)\omega(\{y\in \mathbb{R}^n: M_\delta f(y)>\lambda\})\leq C \sup_{\lambda>0}\varphi(\lambda)\omega(\{y\in \mathbb{R}^n: M^\sharp_\delta f(y)>\lambda\})
\end{equation}
for every function such that the left-hand side is finite.
\par
(2) Let $0<p<\infty.$ There exists a positive constant C depending upon the $A_\infty$ condition and $p$ such that
\begin{equation}
\int_{\mathbb{R}^n}(M_\delta f(x))^p\omega(x)dx\leq C \int_{\mathbb{R}^n}(M^\sharp_\delta f(x))^p\omega(x)dx
\end{equation}
for every function such that the left-hand side is finite.
\end{lemma}
\begin{lemma}\label{23} Let $T$ be a
multilinear square function with a kernel satisfying conditions (H1), (H2) and
(H3) for some $1\leq p_0<\infty.$ For any $0<\delta<\min\{1,\frac{p_0}{m}\}$,
there is a constant $C<\infty$  such that for any bounded and compactly
supported $f_j, (j=1,\dots,m).$

$$
M^{\sharp}_\delta T(\vec{f})(x)\leq C  \mathcal{M}_{p_0}(\vec{f})(x) .
$$

\end{lemma}
\begin{proof}
Fix a point $x\in \mathbb{R}^n$ and a ball $Q$ containing $x$.
For $0<\delta<\min\{1,\frac{p_0}{m}\},$ we only need to show that there exists a
constant $c_Q$ such that
$$
\big( \frac{1}{|Q|}\int_{Q}\big| T(\vec{f})(z)-c_Q\big|^\delta dz
\big)^{1/\delta}\leq C \mathcal{M }_{p_0}\vec{f}(x).
$$
For each $j=1,\dots, m,$ we decompose $f_j=f_j^0+f_j^\infty$,
where $f_j^0=f_j\chi_{Q^*}$, $Q^*$ is the ball with center at $x$ and
having eight times bigger radius than $Q$.


First, we claim that
\begin{equation*}
\bigl( \int_0^\infty \Bigl|\int_{\mathbb{R}^{nm}}K_t(x,y_1,\cdots,y_m)
\prod_{j=1}^m f_j^{\alpha_j}(y_j)d\vec{y}\Bigr|^2\frac{dt}{t} \bigr)^{1/2}
<\infty, \quad\quad \hbox{for\ }\vec\alpha\ne \vec 0.
\end{equation*}
In fact,
set
\begin{equation*}
c_{Q,\alpha}
=\bigl( \int_0^\infty \Bigl|\int_{\mathbb{R}^{nm}}K_t(x,y_1,\cdots,y_m)
\prod_{j=1}^m f_j^{\alpha_j}(y_j)d\vec{y}\Bigr|^2\frac{dt}{t} \bigr)^{1/2}.
\end{equation*}

By the Minkowski inequality, we have
\begin{equation*}
c_{Q,\alpha}
\le \int_{\mathbb{R}^{nm}}\bigl( \int_0^\infty \Bigl|K_t(x,y_1,\cdots,y_m)
\Bigr|^2\frac{dt}{t} \bigr)^{1/2}\prod_{j=1}^m|f_j^{\alpha_j}(y_j)|d\vec{y}.
\end{equation*}
To estimate $c_{Q,\alpha}$, we may assume $\alpha_1=\dots=\alpha_l=0$ and
$\alpha_{l+1}=\dots=\alpha_m=\infty$.
Since $\vec f_j\in L_c^\infty(\mathbb R^n)$,
there exists the smallest $j_0\in\mathbb N$ satisfying
$\operatorname{supp}\vec f\subset 2^{j_0}Q^*$. Then,
by using the H\"older inequality and condition (H3), one may obtain that
\begin{align*}
c_{Q,\alpha}&\le
\bigl(\int_{(2^{j_0}Q^*\setminus Q^*)^{m-l}\times (S_0(Q^*))^l}
\bigl( \int_0^\infty \Bigl|K_t(x,\vec y)\Bigr|^2\frac{dt}{t} \bigr)^{p_0'/2}
d\vec y \bigr)^{1/p_0'}
\\
&\hspace{4cm}\times
\bigl(\int_{(2^{j_0}Q^*\setminus Q^*)^{m-l}\times (S_0(Q^*))^l}
\prod_{j=1}^m|f_j^{\alpha_j}(y_j)|^{p_0}d\vec{y}\bigr)^{1/p_0}
\\
&\le \sum_{k=1}^{j_0}
\bigl(\int_{(S_k(Q^*))^{m-l}\times (S_0(Q^*))^l}
\bigl( \int_0^\infty \Bigl|K_t(x,\vec y)\Bigr|^2\frac{dt}{t} \bigr)^{\frac{p_0'}2}
d\vec y \bigr)^{\frac{1}{p_0'}}
\bigl(\int_{\mathbb R^{nm}}
\prod_{j=1}^m|f_j^{\alpha_j}(y_j)|^{p_0}d\vec{y}\bigr)^{\frac1{p_0}}
\\
&\le C\sum_{k=1}^{j_0} \frac{2^{-nmk/p_0}}{|Q|^{m/p_0}}
\bigl(\int_{\mathbb R^{nm}}
\prod_{j=1}^m|f_j^{\alpha_j}(y_j)|^{p_0}d\vec{y}\bigr)^{\frac1{p_0}}
<\infty.
\end{align*}

Let $c_{Q,t}=\sum_{\vec{\alpha},\vec{\alpha}\neq \vec{0}}
\int_{\mathbb{R}^{nm}}K_t(x,y_1,\cdots,y_m)
\prod_{j=1}^m f_j^{\alpha_j}(y_j)d\vec{y}$ and $c_Q=\big( \int_0^\infty |c_{Q,t}|^2\frac{dt}{t} \big)^{1/2},$
where $\vec{\alpha}=(\alpha_1,\cdots,\alpha_m)$ with $\alpha_i=0$ or $\infty$.
 Then we have
$$\aligned
{} &\big( \frac{1}{|Q|}\int_{Q}\big| T(\vec{f})(z)-c_Q\big|^\delta dz
 \big)^{1/\delta}
\\
&\le  C \big( \frac{1}{|Q|}\int_{Q}\big(\int_0^\infty\big|
\int_{\mathbb{R}^{nm}}K_t(z,\vec{y})\prod_{j=1}^m f_j^0(y_j)d\vec{y}
\big|^2\frac{dt}{t}\big)^{\delta/2} dz\big)^{1/\delta}
\\
&\quad +  C \sum_{\vec{\alpha},\vec{\alpha}\neq \vec{0}}
\big( \frac{1}{|Q|}\int_{Q}\big(\int_0^\infty\big| \int_{\mathbb{R}^{nm}}
\big(K_t(z,\vec{y})-K_t(x,\vec{y})\big)
\prod_{j=1}^m f_j^{\alpha_j}(y_j)d\vec{y}\big|^2\frac{dt}{t}
\big)^{\delta/2} dz\big)^{1/\delta}
\\
&= I_{\vec{0}} + C\sum_{\vec{\alpha}\neq \vec{0}} I_{\vec{\alpha}}.
\endaligned
$$

By condition (H1), $T$ maps $L^{p_0}\times \cdots \times L^{p_0}$
into $L^{p_0/m, \infty}$. This together with the Kolmogorov inequality tells
us that
$$
I_{\vec{0}}\leq C  ||T(f^0)||_{L^{p_0/m,\infty}(Q,\frac{dx}{|Q|})}
\leq C \prod_{j=1}^m \big( \frac{1}{|Q^*|}\int_{Q^*}|f_j(z)|^{p_0}dz
\big)^{p_0}\leq C  \mathcal{M}_{p_0}(\vec{f})(x).
$$
\par
To estimate $I_{\vec{\alpha}}$ for $\vec{\alpha}\neq \vec{0}$, we may assume
$\alpha_1=\dots=\alpha_l=\infty$ and $\alpha_{l+1}=\dots=\alpha_m=0$.
By condition (H2), it yields that
\begin{align*}
&\int_{\mathbb{R}^{nm}}\big(\int_0^\infty
\big| K_t(z,\vec{y})-K_t(x,\vec{y})\big|^2\frac{dt}{t}\big)^{1/ 2}
\prod_{j=1}^m| f_j^{\alpha_j}(y_j)|d\vec{y}\\
  &\leq C \int_{ ((Q^*)^c)^l\times  (Q^*)^{m-l}} \big(\int_0^\infty
   \big| K_t(z,\vec{y})-K_t(x,\vec{y})\big|^2\frac{dt}{t}\big)^{\frac1 2}
   \prod_{j=1}^l \big|f_j^\infty(y_j)\big|dy_j\prod_{j=l+1}^m
   \big|f_j^0(y_j)\big|dy_j\\
 &\leq C \sum_{j_1,\dots, j_l\geq 1}\int_{(Q^*)^{m-l}}\int_{S_{j_l}(Q^*)}
 \cdots \int_{S_{j_1}(Q^*)}\big(\int_0^\infty
  \big| K_t(z,\vec{y})-K_t(x,\vec{y})\big|^2\frac{dt}{t}\big)^{\frac1 2}\\
      &\quad \quad\times \prod_{j=1}^l \big|f_j^\infty(y_j)\big|dy_j
     \prod_{j=l+1}^m  \big|f_j^0(y_j)\big|dy_j\\
          &\leq C \sum_{j_1,\dots, j_l\geq 1}\big(\int_{(Q^*)^{m-l}}
      \int_{S_{j_l}(Q^*)}\cdots  \int_{S_{j_1}(Q^*)}\big(\int_0^\infty
      \big| K_t(z,\vec{y})-K_t(x,\vec{y})\big|^2\frac{dt}{t}
      \big)^{p'_0/ 2}d\vec{y}\big)^{1/p'_0}\\
 &\quad \quad\times \prod_{j=1}^l \big(\int_{2^{j_k}Q^*}|f_j(y_j)|^{p_0}dy_j
 \big)^{1/ p_0}\prod_{j=l+1}^m  \big(\int_{Q^*}|f_j(y_j)|^{p_0}dy_j
 \big)^{1/ p_0}\\
 &\leq  C\sum_{j_1,\dots, j_l\geq 1}
 \frac{|x-z|^{m(\delta-n/p_0)}}{|Q^*|^{m\delta/n}}2^{-m\delta j_0}\\
 &\quad \quad\times \prod_{j=1}^l \big(\int_{2^{j_k}Q^*}|f_j(y_j)|^{p_0}dy_j
 \big)^{1/ p_0}\prod_{j=l+1}^m  \big(\int_{Q^*}|f_j(y_j)|^{p_0}dy_j
 \big)^{1/ p_0}\\
 &\leq C\sum_{j_0 \geq 1} \frac{|x-z|^{m(\delta-n/p_0)}}{|Q^*|^{m\delta/n}}
 m2^{-m\delta j_0}2^{j_0mn/p_0}|Q^*|^{m/p_0}
 \prod_{j=1}^m  \big(\frac{1}{|2^{j_0}Q^*|}\int_{2^{j_0}Q^*}|f_j(y_j)|^{p_0}
 dy_j \big)^{\frac{1}{p_0}}\\
 &\leq
 C\sum_{j_0 \geq 1}\frac{|x-z|^{m(\delta-n/p_0)}}{|Q^*|^{m(\delta/n-1/p_0)}}
 m2^{-mj_0(\delta-n/p_0)} \mathcal{M}_{p_0}(\vec{f})(x)\\
 &\leq C \mathcal{M}_{p_0}(\vec{f})(x).
\end{align*}
 Here, we use the condition $\delta>n/p_0$ and $x,z\in Q$, where $j_0=\max\{j_1,\dots,j_l\}.$
Then, by the Minkowski inequality, we get
\begin{eqnarray*}
  I_{\vec{\alpha}}&\leq & C\big( \frac{1}{|Q|}\int_{Q}\big( \int_{\mathbb{R}^{nm}}\big(\int_0^\infty\big| K_t(z,\vec{y})-K_t(x,\vec{y})\big|^2\frac{dt}{t}\big)^{1/ 2}\prod_{j=1}^m| f_j^{\alpha_j}(y_j)|d\vec{y}\big)^{\delta} dz\big)^{1/\delta}\\
  &\leq & C \mathcal{M}_{p_0}(\vec{f})(x).
\end{eqnarray*}
Thus, we finish the proof of Lemma \ref{23}.
\end{proof}

%
\begin{lemma}\label{22-2}
Suppose $K_t$ satisfies $(H3)$ for some $1\le p_0<\infty$. Suppose $f_i\in C_c^\infty(\mathbb{R}^n)$ and
$\operatorname{supp}f_i \subset B(0,R)$ for any $i=1,\cdots,m.$
 Then there is a constant $C<\infty$  such that for $|x|>3R,$
 the following estimate holds uniformly.
$$
T(\vec{f})(x)\leq C  \mathcal{M }_{p_0}(\vec{f})(x).
$$

\end{lemma}
\begin{proof}
By the Minkowski inequality, the H\"older inequality, the support property of
$f_j$ and condition (H3), we obtain
\begin{align*}
T(\vec f)(x)&\le \int_{B(0,R)^m}
\bigl( \int_{0}^\infty |K_t(x,y_1,\cdots,y_m)|^2 \frac{dt}{t}
\bigr)^{\frac {1}2}\prod_{j=1}^m|f_{j}(y_j)|d\vec y
\\
&\le \int_{(B(x,\frac43 |x|)\setminus (B(x,\frac23 |x|))^m}
\bigl( \int_{0}^\infty |K_t(x,y_1,\cdots,y_m)|^2 \frac{dt}{t}
\bigr)^{\frac {1}2}\prod_{j=1}^m|f_{j}(y_j)|d\vec y
\\
&\le \bigl(\int_{(B(x,\frac43 |x|)\setminus (B(x,\frac23 |x|))^m}
\bigl( \int_{0}^\infty |K_t(x,y_1,\cdots,y_m)|^2 \frac{dt}{t}
\bigr)^{\frac {p_0'}2}d\vec y\bigr)^{1/p_0'}
\\
&\quad\quad\times
\bigl(\int_{(B(x,\frac43 |x|)\setminus (B(x,\frac23 |x|))^m}
\prod_{j=1}^m|f_{j}(y_j)|^{p_0}d\vec y\bigr)^{1/p_0}
\\
&\le C\mathcal{M}_{p_0}(\vec{f})(x),
\end{align*}
which concludes the proof of Lemma \ref{22-2}.
\end{proof}

\textbf{Proof of Theorem \ref{1}.}
 First, we will show that (1) in Theorem \ref{1} is true. By Lemma \ref{22}, we may assume that $||\mathcal{M}_{p_0}\vec{f}||_{L^p(\nu_{\vec{\omega}})}$ is finite. Without loss of generality, we futher assume that each $f_i>0$, $f_i\in C_c^\infty(\mathbb{R}^n)$ and $\nu_{\vec{\omega}}$ are bounded functions. Now, we claim that
$\int_{\mathbb{R}^n}(T(\vec{f}))^p\nu_{\vec{\omega}}dx< \infty.$
In fact,
 $$\int_{\mathbb{R}^n}(T(\vec{f}))^p\nu_{\vec{\omega}}dx=\int_{3B}(T(\vec{f}))^p\nu_{\vec{\omega}}dx+\int_{(3B)^c}(T(\vec{f}))^p\nu_{\vec{\omega}}dx.$$

By the H\"{o}lder inequality, it yields that
$$\int_{3B}(T(\vec{f}))^p\nu_{\vec{\omega}}dx\leq C \prod
\limits_{i=1}^{m}||f_i||_{L^{p_i} }<\infty.$$
On the one hand, by using Lemma \ref{22-2}, it holds that
$$\int_{(3B)^c}(T(\vec{f}))^p\nu_{\vec{\omega}}dx\leq  \int_{(3B)^c}(\mathcal{M}_{p_0}(\vec{f}))^p\nu_{\vec{\omega}}dx<\infty.$$

Now, we are in a position to prove  $\int_{\mathbb{R}^n}(M_\delta T(\vec{f}))^p\nu_{\vec{\omega}}dx<\infty.$  Since $\omega\in A_\infty,$ then there exists $q_0>1,$ such that $\omega\in A_{q_0}.$ We may take $\delta>0$, small enough and $p/ \delta>q_0$ such that $\omega\in A_{p/ \delta}.$ Then, the boundedness of $M$ yields that
$$\int_{\mathbb{R}^n}(M_\delta T(\vec{f}))^p\nu_{\vec{\omega}}dx<\int_{\mathbb{R}^n}( T(\vec{f}))^p\nu_{\vec{\omega}}dx< \infty.$$

Thus, the desired estimates follows by using Fefferman-Stein's inequality,

\begin{eqnarray*}
 \big(\int_{\mathbb{R}^n}\big(T(\vec{f})\big)^p\nu_{\vec{\omega}}dx\big)^{1/p}\leq C  \big( \int_{\mathbb{R}^n}(M_\delta \big(T(\vec{f})\big))^p\nu_{\vec{\omega}}dx\big)^{1/p}\leq C  \big( \int_{\mathbb{R}^n}(M_\delta^\sharp \big
 (T(\vec{f})\big))^p\nu_{\vec{\omega}}dx\big)^{1/p}\\
 \leq C \big( \int_{\mathbb{R}^n}(\mathcal{M}_{p_0}(\vec{f}))^p\nu_{\vec{\omega}}dx\big)^{1/ p}\leq C \big( \int_{\mathbb{R}^n} |f_i|^{p_i}\omega_idx\big)^{1/ p_i}.
\end{eqnarray*}

The proof of Theorem \ref{1} (2) can be treated similarly as that in Theorem \ref{1} (1), with only a slight modifications. Thus, we omit the proof of it.

Thus we complete the proof of Theorem \ref{1}. \qed
\par\smallskip

Next, we turn to the proof of Theorem \ref{2}. We prepare several lemmas.
\begin{lemma} Suppose $K_t$ satisfies $(H3)$ for some $1\le p_0<\infty$. Suppose $f_i\in C_c^\infty(\mathbb{R}^n)$ and
$\operatorname{supp}f_i \subset B(0,R)$ for any $i=1,\cdots,m.$
Then there is a constant $C<\infty$  such that for $|x|>3R$
and bounded function $b_j(x)$, $j=1,\cdots,m$, the following estimate holds uniformly.
$$
T_{\vec{b}}(\vec{f})(x)\leq C\|\vec b\|_\infty \mathcal{M }_{p_0}(\vec{f})(x).
$$
\end{lemma}
\begin{proof}
We can use the same arguments as in Lemma \ref{22-2} to finish the proof.
\end{proof}
\begin{lemma}\label{lem31} Let $T$ be a
multilinear square function with a kernel satisfying conditions (H1), (H2) and (H3) for some $1\leq p_0<\infty.$
Then, for any $0<\delta<\varepsilon<\min\{1,\frac{p_0}{m}\}$ and $q_0>p_0$,
there is a constant $C<\infty$
such that for any bounded and compactly supported $f_j (j=1,\dots,m),$ the following inequality holds
$$
M^{\sharp}_\delta \big(T_{\vec{b}}(\vec{f})\big)(x)
\leq C  ||\vec{b}||_{BMO}\big(\mathcal{M}_{q_0}(\vec{f})(x)+M_\varepsilon
\big(T (\vec{f})\big)(x)\big) .
$$
\end{lemma}
\begin{proof}
 We may assume
$\vec b=(b,0,\dots,0)$.
Fix a point $x\in \mathbb{R}^n$ and a ball $Q$ containing $x$.
For $0<\delta<\varepsilon<\min\{1,\frac{p_0}{m}\}$,
we need to show that there exists a constant $c_Q$ such that
$$
\bigg( \frac{1}{|Q|}\int_{Q}\big| T_{\vec{b}}(\vec{f})(z)-c_Q\big|^\delta dz
\bigg)^{1/\delta}
\leq C ||\vec{b}||_{BMO}\big(\mathcal{M}_{q_0}(\vec{f})(x)+M_\varepsilon
\big(T (\vec{f})\big)(x)\big).
$$
For any constant $c_Q$, we have
\begin{align*}
 &\Big( \frac{1}{|Q|}\int_{Q}
 \big| T_{\vec{b}}(\vec{f})(z)-c_Q\big|^\delta dz\Big)^{1/\delta}
\\
&\leq  C \Big( \frac{1}{|Q|}\int_{Q}|b(z)-b_{Q^*}|^\delta\Bigl(
\int_0^\infty\Big|\int_{\mathbb{R}^{nm}}K_t(z,\vec{y})\prod_{j=1}^m f_j(y_j)
d\vec{y}\Big|^2\frac{dt}{t}\Big)^{\delta/2} dz\Big)^{1/\delta}
\\
&\quad +  C \Big( \frac{1}{|Q|}\int_{Q}\Big|\Big(\int_0^\infty\Bigl|
\int_{\mathbb{R}^{nm}}K_t(z,\vec{y})
(b(y_1)-b_{Q^*})\prod_{j=1}^m f_j(y_j)d\vec{y}\Big|^2\frac{dt}{t}
\Big)^{1/2}%
-c_Q\Big|^{\delta} dz\Big)^{1/\delta}
\\&
:= I  +  II.
\end{align*}

The H\"{o}lder inequality gives that
\begin{eqnarray*}
 I &\leq & C
\Big( \frac{1}{|Q^*|}\int_{Q^*}|(b(z)-b_{Q^*})|^{p'\delta}dz\Big)^{1/ p'\delta}
\Big( \frac{1}{|Q^*|}\int_{Q^*}|T(\vec{f})(z)|^{p\delta}dz\Big)^{1/ p\delta}
\\
 &\leq & C ||\vec{b}||_{BMO} M_\varepsilon \big(T (\vec{f})\big)(x),
 \end{eqnarray*}
where we have chosen $p>1$ so that $\delta p<\varepsilon<p_0/m$ and
$\delta p'>1$.

Now for each $j$ we decompose $f_j=f^0_j+f^\infty_j$,
 where $f^0_j=f_j\chi_{Q^*},j=1,\dots, m,$ and $Q^*=8Q.$
 Then
$$
\prod_{j=1}^m f_j(y_j)=\sum_{\vec{\alpha}}f_1^{\alpha_1}(y_1)\cdots
f_m^{\alpha_m}(y_m),
$$
where $\vec{\alpha}=(\alpha_1,\cdots,\alpha_m)$ with $\alpha_i=0$ or $\infty.$

Now, we introduce the notion, $c_Q=\big( \int_0^\infty |c_{Q,t}|^2\frac{dt}{t} \big)^{1/2},$
where $$c_{Q,t}=  \sum_{\vec{\alpha}\neq \vec{0}}
\int_{\mathbb{R}^{nm}}(b(y_1)-b_{Q^*})K_t(x,y_1,\cdots,y_m)
\prod_{j=1}^m f_j^{\alpha_j}(y_j)d\vec{y}.$$
Similarly as before, the finiteness of $c_Q$ follows from the condition (H3). Moreover,
\begin{align*}
II&\le\big( \frac{1}{|Q|}\int_{Q}\big(\int_0^\infty\big|
\int_{\mathbb{R}^{nm}}K_t(z,\vec{y})(b(y_1)-b_{Q^*})
\prod_{j=1}^m f_j(y_j)d\vec{y}-c_{Q,t}\big|^2\frac{dt}{t}\big)^{\delta/2}
dz\big)^{1/\delta}
\\
&\leq C\big( \frac{1}{|Q|}\int_{Q}\big(\int_0^\infty\big|
\int_{\mathbb{R}^{nm}}K_t(z,\vec{y})(b(y_1)-b_{Q^*})\prod_{j=1}^m f_j^0(y_j)
d\vec{y} \big|^2\frac{dt}{t}\big)^{\delta/2} dz\big)^{1/\delta}
\\
&+ \sum_{\vec{\alpha}\neq \vec{0}}
\big( \frac{1}{|Q|}\int_{Q}\big(\int_0^\infty\big| \int_{\mathbb{R}^{nm}}
\big(K_t(z,\vec{y})-K_t(x,\vec{y})\big)(b(y_1)-b_{Q^*})\\&\quad\times
\prod_{j=1}^m |f_j^{\alpha_j}(y_j)|d\vec{y}\big|^2\frac{dt}{t}
\big)^{\delta/2} dz\big)^{1/\delta}
:=II_{\vec{0} } + II_{\vec{\alpha}:\vec{\alpha}\neq \vec{0}}.
\end{align*}
The condition (H1), together with the Kolmogorov inequality ($p_0<q_0$) gives that
\begin{align*}
 II_{\vec{0} }
&\leq  C \big( \frac{1}{|Q|}
\int_{Q}|T((b-b_{Q^*})f^0_{1},\dots,f^0_m)(z)|^{\delta}dz\big)^{1/ \delta}
\\
 &\leq  C
   ||T((b-b_{Q^*})f^0_{1},\dots,f^0_m)||_{L^{p_0/m,\infty}(Q,\frac{dx}{|Q|})}
\\
   &\leq  C \big( \frac{1}{|Q|}\int_{Q}|(b(z)-b_{Q^*})f^0_1(z)|^{p_0}dz
   \big)^{1/ p_0} \prod_{j=2}^\infty \big( \frac{1}{|Q|}
   \int_{Q}|f_j^0(z)|^{p_0}dz\big)^{1/ p_0}
\\
    &\leq  C ||\vec{b}||_{BMO} \mathcal{M}_{q_0}(\vec{f})(x).
 \end{align*}
A similar argument as in the proof of Lemma \ref{23} will lead to that
 \begin{align*}
&\int_{\mathbb{R}^{nm}}\big(\int_0^\infty\big| K_t(z,\vec{y})-K_t(x,\vec{y})
\big|^2\frac{dt}{t}\big)^{1/ 2}|(b(y_1)-b_{Q^*})|
\prod_{j=1}^m| f_j^{\alpha_j}(y_j)|d\vec{y}
\\
  &\leq C\sum_{j_0 \geq 1} \frac{|x-z|^{m(\delta-n/p_0)}}{|Q^*|^{m\delta/n}}
  m2^{-m\delta j_0}2^{j_0mn/p_0}|Q^*|^{m/p_0}
\\
 &\times \big(\frac{1}{|2^{j_0}Q^*|}
 \int_{2^{j_0}Q^*}|(b(y_1)-b_{Q^*})f_1(y_1)|^{p_0}dy_1 \big)^{\frac1{ p_0}}
 \prod_{j=2}^m  \big(\frac{1}{|2^{j_0}Q^*|}\int_{2^{j_0}Q^*}|f_j(y_j)|^{p_0}
 dy_j \big)^{\frac1{ p_0}}
\\
 &\leq  C ||\vec{b}||_{BMO} \mathcal{M}_{q_0}(\vec{f})(x).
\end{align*}
 Here, $\delta>n/p_0$ and $x,z\in Q$.
Then,  by Minkowski's inequality, we get
\begin{align*}
  II_{\vec{\alpha}:\vec{\alpha}\neq \vec{0}}
  &\leq  C\big( \frac{1}{|Q|}\int_{Q}\big( \int_{\mathbb{R}^{nm}}
\big(\int_0^\infty\big| K_t(z,\vec{y})-K_t(x,\vec{y})\big|^2\frac{dt}{t}
\big)^{1/ 2}|(b(y_1)-b_{Q^*})|\\&\quad \times\prod_{j=1}^m| f_j^{\alpha_j}(y_j)|d\vec{y}
\big)^{\delta} dz\big)^{1/\delta}
\\
  &\leq  C \mathcal{M}_{q_0}(\vec{f})(x).
\end{align*}
Then, the proof of Lemma \ref{lem31} is finished.
\end{proof}
\textbf{Proof of Theorem \ref{2}.} We may assume that $||\vec{b}||_{BMO}=1$. By repeating the same arguments as in the proof of Theorem \ref{1} (1), we get $\int_{\mathbb{R}^n}(T_{\vec{b}}(\vec{f}))^p\nu_{\vec{\omega}}dx$ and
$\int_{\mathbb{R}^n}(M_\delta T_{\vec{b}}(\vec{f}))^p\nu_{\vec{\omega}}dx$ are finite.
Since $\nu_{\vec{\omega}}\in A_{pm/p_0}$, Theorem \ref{1} (1) gives that
$$\big( \int_{\mathbb{R}^n}(M_\epsilon \big(T (\vec{f})\big))^p\nu_{\vec{\omega}}dx\big)^{1/p}\leq C\big( \int_{\mathbb{R}^n}\big(T (\vec{f})\big)^p\nu_{\vec{\omega}}dx\big)^{1/p}\leq C \big( \int_{\mathbb{R}^n} |f_i|^{p_i}\omega_idx\big)^{1/ p_i}.$$
It is known from \cite{lerner} that if $\vec{\omega}\in A_{\vec{P}/ p_0,}$ then there exists $q_0>p_0$ such that $\vec{\omega}\in A_{\vec{P}/ q_0}.$ Then
$$\|\mathcal{M}_{q_0}(\vec{f}) \|_{L^{p}(\nu_{\vec{w}})}\leq C \prod_{i=1}^{m} \|f_i\|_{L^{p_i}(w_i)} .$$
Thus, we have
\begin{eqnarray*}
 \big( \int_{\mathbb{R}^n}\big(T_{\vec{b}}(\vec{f})\big)^p\nu_{\vec{\omega}}dx\big)^{1/p}\leq C  \big( \int_{\mathbb{R}^n}(M_\delta \big(T_{\vec{b}}(\vec{f})\big))^p\nu_{\vec{\omega}}dx\big)^{1/p}\leq C  \big( \int_{\mathbb{R}^n}(M_\delta^\sharp \big(T_{\vec{b}}(\vec{f})\big))^p\nu_{\vec{\omega}}dx\big)^{1/p}\\
 \leq C \big( \int_{\mathbb{R}^n}(\mathcal{M}_{q_0}(\vec{f}))^p\nu_{\vec{\omega}}dx\big)^{1/ p}+C\big( \int_{\mathbb{R}^n}(M_\epsilon \big(T (\vec{f})\big))^p\nu_{\vec{\omega}}dx\big)^{1/p}\leq C \big( \int_{\mathbb{R}^n} |f_i|^{p_i}\omega_idx\big)^{1/ p_i}.
\end{eqnarray*}
Hence, we complete the proof of Theorem \ref{2}.

\section{Proof of Theorem 1.5}\label{Sec-5}
We begin with some basic lemmas.
\begin{lemma}\label{lem:LlogL^a}
For $0<\alpha<\infty$, let
\begin{equation*}
\Phi(t)=t(1+\log^+t)^{\alpha}, \quad 0<t<\infty.
\end{equation*}
Then it is a Young function and its complementary Young function is equivalent
to
\begin{equation*}
\overline{\Phi_1}(t)=\int_0^t \bar\varphi_1(s)ds,\ \text{ where }
\bar\varphi_1(t)=\begin{cases}
t^{1/\alpha}, &0<t<1
\\ e^{t^{1/\alpha}-1}, &t\ge1.
\end{cases}
\end{equation*}
\end{lemma}
\begin{proof}
Let
\begin{equation*}
\Phi_0(t)=\begin{cases}
t^{1+\alpha}, &0<t<1,
\\
t(1+\log t)^{\alpha}, & 1\leq t<\infty.
\end{cases}
\end{equation*}
Then $\Phi_0(t)\sim\Phi(t)$ and
\begin{equation*}
\phi_0(t)=\Phi_0'(t)=\begin{cases}
(1+\alpha)t^{\alpha}, &0<t<1,
\\
(1+\log t)^{\alpha}+\alpha(1+\log t)^{\alpha-1}, &1<t<\infty.
\end{cases}
\end{equation*}
Futhermore,
\begin{equation*}
\phi_0'(t)=\begin{cases}
\alpha(1+\alpha)t^{\alpha-1}, &0<t<1,
\\
\frac{\alpha(1+\log t)^{\alpha-1}}{t}
+\frac{\alpha(\alpha-1)(1+\log t)^{\alpha-2}}{t}, & 1<t<\infty.
\end{cases}
\end{equation*}
So, $\Phi_0(t)$ is also a Young function.
Let
\begin{equation*}
\phi_1(t)=\begin{cases}
t^{\alpha}, &0<t<1,
\\
(1+\log t)^{\alpha}, &1\leq t<\infty.
\end{cases}
\end{equation*}
Then we see that $\phi_1(t)<\phi_0(t)\le (1+\alpha)\phi_1(t)$ $(0<t<\infty)$
and
\begin{equation*}
\phi_1'(t)=\begin{cases}
\alpha t^{\alpha-1}, &0<t<1,
\\
\frac{\alpha(1+\log t)^{\alpha-1}}{t}, & 1<t<\infty.
\end{cases}
\end{equation*}
Hence $\Phi_1(t)=\int_0^t \phi_1(s)ds$ is a Young function and
is equivalent to $\Phi_0(t)$ and so to $\Phi(t)$.

The inverse function of $\phi_1(t)$ is given by
\begin{equation*}
\bar\phi_1(t)=\begin{cases}
t^{1/\alpha}, &0<t<1,\\ e^{t^{1/\alpha}-1}, &t\ge1,
\end{cases}
\end{equation*}
which completes the proof of Lemma \ref{lem:LlogL^a}.

\end{proof}
\begin{lemma}\label{lem:p-BMO-Lp-duaity}
Let $b\in \mathrm{BMO}(\mathbb R^n)$ and $f\in L^p\log^pL(\mathbb R^n)$ for
some $1\le p<\infty$. Then,  for any ball $Q$,
$x\in Q$ and $j\in\mathbb N_0$, there exists a constant $C>0$ such that
\begin{equation*}
\Big( \frac{1}{|2^jQ|}\int_{2^jQ}|(b(y)-b_{Q})f(y)|^{p}dy\Big)^{1/p}
\leq  C (j+1)||{b}||_{BMO} {M}_{L^p\log^pL }(f)(x).
\end{equation*}
\end{lemma}
\begin{proof}
(a) The case $j=0$.
Let $Q$ be a ball in $\mathbb R^n$.
Let $\Phi(t)$ and $\overline{\Phi_1}(t)$ be in Lemma \ref{lem:LlogL^a} as
$\alpha=p$.
Then by the H\"older inequality in Orlicz spaces, it holds that
\begin{equation}\label{eq:BMO-Lp-1}
\frac{1}{|Q|}\int_{Q}|(b(y)-b_{Q})f(y)|^{p}dy
\le C\|b\|_{*}^p\|((b(y)-b_Q)/\|b\|_*)^p \|_{{\bar{\Phi}_1,Q}}
\||f|^p\|_{{\Phi,Q}}.
\end{equation}
Note that
\begin{equation*}
\overline{\Phi_1}(t)\le \int_{0}^{t}e^{s^{1/\alpha}}ds=:\Psi(t).
\end{equation*}
Thus, for any $c>0$, we have
\begin{align*}
\frac{1}{|Q|}\int_Q \bar{\Phi}_1
\Bigl(\frac{|b(y)-b_Q|^p}{(c\|b\|_*)^p}\Bigr)dy
&\le
\frac{1}{|Q|}\int_Q \Psi\Bigl(\frac{|b(y)-b_Q|^p}{(c\|b\|_*)^p}\Bigr)dy
\\
&=\frac{1}{|Q|}\int_Q
\int_{0}^{\frac{|b(y)-b_Q|^p}{(c\|b\|_*)^p}}e^{s^{1/p}}ds\,dy
\\
&=\frac{1}{|Q|}\int_Q
\int_{0}^{\infty}\chi_{\{\frac{|b(y)-b_Q|^p}{(c\|b\|_*)^p}>s\}}
e^{s^{1/p}}ds\,dy
\\
&=\int_{0}^{\infty}e^{s^{1/p}}
\biggl( \frac{1}{|Q|}\int_Q \chi_{\{\frac{|b(y)-b_Q|^p}{(c\|b\|_*)^p}>s\}}dy
\biggr)ds.
\end{align*}
On the other hand, by the John-Nirenberg inequality, there exist positive
constants $c_1$ and $c_2$ such that
\begin{equation*}
|\{x\in Q: |b(x)-b_Q|>\lambda\}|\le c_2|Q|e^{-c_1\lambda/\|b\|_*}, \quad
\lambda>0.
\end{equation*}
Hence, choosing $c$ big enough such that $c_1c>1$, we get
\begin{equation*}
\frac{1}{|Q|}\int_Q \bar{\Phi}_1
\Bigl(\frac{|b(y)-b_Q|^p}{(c\|b\|_*)^p}\Bigr)
\le
\int_{0}^{\infty}e^{s^{1/p}}c_2e^{-c_1cs^{1/p}}ds
=c_2\int_{0}^{\infty}e^{-(c_1c-1)s^{1/p}}ds<\infty,
\end{equation*}
which shows that the norm $\|((b(y)-b_Q)/\|b\|_*)^p \|_{{\bar{\Phi}_1},Q}$
is bounded by a constant depending on $p,c_1, c_2$.
Combining this with \eqref{eq:BMO-Lp-1} gives
\begin{equation*}
\Big( \frac{1}{|Q|}\int_{Q}|(b(y)-b_{Q})f(y)|^{p}dy\Big)^{1/p}
\leq  C \|{b}\|_{BMO} {M}_{L^p\log^pL }(f)(x),
\end{equation*}
for any $x\in Q$.
\par\noindent
(b) The case $j\in\mathbb N$. By the Minkowski inequality and step (a), one obtains
\begin{align*}
&\Big(\frac{1}{|2^jQ|}\int_{2^jQ}|(b(y)-b_{Q})f(y)|^{p}dy\Big)^{1/p}
\\
&\le \sum_{l=1}^{j}
\Big( \frac{1}{|2^jQ|}\int_{2^jQ}|(b_{2^lQ}-b_{2^{l-1}Q})f(y)|^{p}dy\Big)^{1/p}
+\Big( \frac{1}{|2^jQ|}\int_{2^jQ}|(b(y)-b_{2^jQ})f(y)|^{p}dy\Big)^{1/p}
\\
&\le \sum_{l=1}^{j}\frac{2^n}{|2^{l}Q|}\int_{2^{l}Q}|(b(y)-b_{2^{l}Q})|dy
\Big(\frac{1}{|2^{j}Q|}\int_{2^{j}Q}|f(y)|^{p}dy\Big)^{1/p}
+C||{b}||_{BMO} \||f|^p\|_{{\Phi,Q}}
\\
&\le C(j+1)||{b}||_{BMO} {M}_{L^p\log^pL }(f)(x).
\end{align*}
This completes the proof of Lemma \ref{lem:p-BMO-Lp-duaity}.

\end{proof}

Using Lemma \ref{lem:p-BMO-Lp-duaity}, we can improve Lemma \ref{lem31} as follows.
\begin{lemma} Let $T$ be a
multilinear square function with a kernel satisfying conditions (H1), (H2)
and (H3) for some $1\leq p_0<\infty.$
Then, for any $0<\delta<\varepsilon<\min\{1,\frac{p_0}{m}\}$,
there is a constant $C<\infty$
such that for any bounded and compactly supported $f_j, j=1,\dots,m.$
$$
M^{\sharp}_\delta \big(T_{\vec{b}}(\vec{f})\big)(x)
\leq C  ||\vec{b}||_{BMO}\big(\sum_{i=1}^m\mathcal{M}_{\Phi}^{(i)}(\vec{f})(x)
+M_\varepsilon\big(T (\vec{f})\big)(x)\big),
$$
where $\Phi(t)=t^{p_0}(1+\log^+ t)^{p_0}$ and
\begin{equation*}
\mathcal{M}_{\Phi}^{(i)}(\vec{f})(x)=\sup_{Q\ni x}\|f_i\|_{\Phi,Q}
\prod_{j\ne i}\left(\frac{1}{|Q|}\int_Q |f_j(y)|^{p_0}\,dy\right)^{\frac 1 p_0}.
\end{equation*}
\end{lemma}
The proof of Theorem \ref{3} will be based on the following lemmas.
\begin{lemma}\label{4}
Let $T$ be a multilinear square function with a kernel satisfying condition
condition  (H1), (H2) and (H3) for some $1\le p_0<\infty$.
 Let $\omega$ be an $A_\infty$ weight and
 let $\Phi(t)=t^{p_0}(1+\log^+ t)^{p_0}$. Suppose that $\vec{b}\in BMO^m.$
 Then,  there exists a constant $C$ (independent of $\vec{b}$) such that the following inequality holds
\begin{align*}
  \sup_{t>0} \frac{1}{\Phi(1/ t)}
  \omega(\{y\in \mathbb{R}^n: |T_{\vec{b}}\vec{f}(y)|&>t^m \})
\\
  &\leq C \sum_{i=1}^m\sup_{t>0} \frac{1}{\Phi(1/ t)}
  \omega(\{y\in \mathbb{R}^n: |\mathcal{M}_{\Phi}^{(i)}\vec{f}(x)|>t^m \}),
\end{align*}
for all bounded vector function $\vec{f}=(f_1,\cdots, f_m)$ with compact
support.
\end{lemma}
\begin{proof}
We borrow some ideas from Theorem 3.19 in \cite{lerner}.  We may assume $||b_j||_{BMO}=1, $ for $j=1,\cdots,m.$ It is enough to  prove the result for
\begin{equation}\label{41}
\begin{split}
T_{b}(\vec{f})(x)
= \bigg( \int_{0}^\infty \Big|
\int_{(\mathbb{R}^n)^m}(b_{1}(x)-b_1(y_1))K_t(x,\vec{y})
\prod_{j=1}^mf_{j}(y_j)dy_1 \dots dy_m\Big|^2\frac{dt}{t}\bigg)^{\frac 12}.
\end{split}
\end{equation}
Let $0<\delta<\varepsilon<1/m$.
By the Lebesgue differentiation theorem, it suffices to prove
\begin{equation}\label{42}
\begin{split}
\sup_{t>0} \frac{1}{\Phi(1/ t)}\omega & (\{y\in \mathbb{R}^n: |M_\delta (T_{\vec{b}}\vec{f})(y)|>t^m\})\\
&\leq C\sup_{t>0}  \frac{1}{\Phi(1/ t)}\omega(\{y\in \mathbb{R}^n:
\mathcal{M}_{\Phi}^{(1)}(\vec{f})(y)>t^m\}).
\end{split}
\end{equation}

However, Lemma \ref{21} yields that
\begin{equation}\label{43}
\begin{split}
\sup_{t>0} \frac{1}{\Phi(1/ t)}
& \omega(\{y\in \mathbb{R}^n: |M_\delta T_{\vec{b}}\vec{f})(y)|>t^m\})
\\
&\leq C\sup_{t>0}\frac{1}{\Phi(1/ t)}
\omega(\{y\in \mathbb{R}^n:|M^\sharp_\delta (T_{\vec{b}}\vec{f})(y)|>t^m\}),
\end{split}
\end{equation}
whenever the left-hand side is finite. Therefore, \eqref{42} follows
from
\begin{equation}\label{44}
\begin{split}
\sup_{t>0} \frac{1}{\Phi(1/ t)}\omega & (\{y\in \mathbb{R}^n: |M^\sharp_\delta (T_{\vec{b}}\vec{f})(y)|>t^m\})\\
&\leq C\sup_{t>0}  \frac{1}{\Phi(1/ t)}\omega(\{y\in \mathbb{R}^n:
\mathcal{M}_{\Phi}^{(1)}(\vec{f})(y)>t^m\}).
\end{split}
\end{equation}

In order to use Fefferman-Stein inequality,
we claim the following inequalities hold:
\begin{equation}\label{441}
\begin{split}
\sup_{t>0} \frac{1}{\Phi(1/ t)}\omega
& (\{y\in \mathbb{R}^n: |M_\delta (T_{\vec{b}}\vec{f})(y)|>t^m\})<\infty.
\end{split}
\end{equation}
and
\begin{equation}\label{442}
\begin{split}
\sup_{t>0} \frac{1}{\Phi(1/ t)}\omega
& (\{y\in \mathbb{R}^n: |M_\varepsilon (T\vec{f})(y)|>t^m\})<\infty.
\end{split}
\end{equation}

Admitting the claim first, we will prove \eqref{44}. Lemma \ref{23}, Lemma \ref{lem31} and
Fefferman-Stein inequality yield that
\begin{equation*}\label{410}
\begin{split}
&\sup_{t>0} \frac{1}{\Phi(1/ t)}
\omega(\{y\in \mathbb{R}^n: |M^\sharp_\delta (T_{\vec{b}}\vec{f})(y)|>t^m\})
\\
&\leq  C\sup_{t>0}  \frac{1}{\Phi(1/ t)}
\omega(\{y\in \mathbb{R}^n:  \mathcal{M}_{\Phi}^{(1)}(\vec{f})(y)
+M_\varepsilon (T \vec{f})(y)>t^m\})
\\
&\leq C\sup_{t>0} \frac{1}{\Phi(1/ t)}
\omega(\{y\in \mathbb{R}^n: \mathcal{M}_{\Phi}^{(1)}(\vec{f})(y)>t^m\})
\\
&\quad+ C\sup_{t>0}  \frac{1}{\Phi(1/ t)}
\omega(\{y\in \mathbb{R}^n: M_\varepsilon (T \vec{f})(y)>t^m\})
\\
&\leq C \sup_{t>0}  \frac{1}{\Phi(1/ t)}
\omega(\{y\in \mathbb{R}^n: \mathcal{M}_{\Phi}^{(1)}(\vec{f})(y)>t^m\})
\\
&\quad + C\sup_{t>0}  \frac{1}{\Phi(1/ t)}
\omega(\{y\in \mathbb{R}^n:\mathcal{M}_{p_0}(\vec{f})(y)>t^m\})
\\
&\leq C \sup_{t>0}  \frac{1}{\Phi(1/ t)}
\omega(\{y\in \mathbb{R}^n: \mathcal{M}_{\Phi}^{(1)}(\vec{f})(y)>t^m\}).\\
\end{split}
\end{equation*}

Now, we only need to show that \eqref{441} holds, by the reason that the proof of \eqref{442} is very
similar but much easier.
We  assume that the $b_j$
and $\omega$ are bounded. Suppose that $\operatorname{supp} f\subset B(0,R).$
Hence, since $\Phi(t)\geq t^{p_0}$ and $0 < \delta< 1/m$, it follows that
\begin{equation*}\label{45}
\begin{split}
&\sup_{t>0} \frac{1}{\Phi(1/ t)}
\omega(\{y\in \mathbb{R}^n: M_{\delta} (|T_{\vec{b}}\vec{f}|)(y)>t^m\})
\\
&\leq C ||\omega||_{L^\infty}\sup_{t>0}\frac{1}{\Phi(1/ t)}\left|
\{y\in \mathbb{R}^n: M_{m\delta} (|T_{\vec{b}}\vec{f}|^{1/ m})(y)>t\}\right|
\\
&\leq C
\sup_{t>0}t^{p_0}
\left|\{y\in \mathbb{R}^n: |T_{\vec{b}}\vec{f}(y)|^{1/ m}>t\}\right|
\\
&\leq C
\sup_{t>0}t^{p_0}\left|\{y\in B_{3R}: |T_{\vec{b}}\vec{f}(y)|^{1/ m}>t\}\right|
+\sup_{t>0}t^{p_0}\left|\{y\in B_{3R}^c: |T_{\vec{b}}\vec{f}(y)|^{1/ m}>t\}
\right|\\
&=I+II.
\end{split}
\end{equation*}
We first consider the contribution of $I$. Taking $r>1$, by the Assumption (H1) and the H\"{o}lder inequality, we have
\begin{equation*}\label{46}
I \leq C \int_{B(0,3R)}|T_{\vec{b}}\vec{f}(y)|^{p_0/ m}dy
 \leq C R^{(1-1/r)n}\Bigl(
\int_{\mathbb{R}^n}\left|T\vec{f}(y)\right|^{p_0r/m}dy\Bigr)^{1/r}<\infty.
\end{equation*}
For the contribution of $II$, note that we may control $|T_{\vec{b}}\vec{f}(x)|$ by
$\mathcal{M}_{p_0}\vec{f}(x) $ if we assume that $b$ is bounded.  Then
\begin{equation*}\label{47}
\begin{split}
II^m& \leq
Ct^{mp_0}|\{y\in \mathbb{R}^n: \mathcal{M}_{p_0}\vec{f}(y)^{1/m}>t\}|^m
\\
& \le\bigl(\|\mathcal{M}_{p_0}\vec{f}\|_{L^{p_0/m,\infty}}\bigr)^{p_0}
\leq C \Bigl(\prod_{i=1}^m \|f_i(x)\|_{p_0}\Bigr)^{p_0}<\infty.
\end{split}
\end{equation*}
Thus,
the claim \eqref{44} is proved. Hence, we finish the proof of Lemma \ref{4}.
\end{proof}
\begin{lemma}\label{5}
 Let $1\le p_0<\infty$ and $\vec{\omega}\in A_{\vec{1}}$.
Then, there exists a constant $C $ such that for $1\le i\le m$
$$
\nu_{\vec{\omega}}\bigl(
\{x\in \mathbb{R}^n: |\mathcal{M}_{\Phi}^{(i)}\vec{f}(x)|>t^m \}\bigr)
\leq C \prod_{j=1}^m \biggl(\int_{\mathbb{R}^n} \Phi \Bigl(\frac{|f_j(x)|}{t}
\Bigr)\omega_j(x)dx\biggr)^{1/ m},
$$
where $\Phi(t)=t^{p_0}(1+\log^+ t)^{p_0}$.
\end{lemma}
\begin{proof}
Some ideas will be taken from the proof of Theorem 3.17 in \cite{lerner}. By homogeneity,
we may assume that $t = 1$ and $\vec{f}\geq 0$. Set
$$
\Omega =\{x\in \mathbb{R}^n: \mathcal{M}_{\Phi}^{(i)}\vec{f}(x) >1\}.
$$

It is easy to see that $\Omega$ is open and we may assume that it is not empty. To estimate the size of $\Omega$, it is
enough to estimate the size of every compact set F contained in $\Omega$.
We note that we may use cubes in place of balls in the definition of maximal
functions.
Now, we can cover any such $F$ by a finite family of cubes
$Q_j$ for which
$$
1<\|f_1\|_{\Phi,Q_i}\prod_{j=2}^m (f_j)_{Q_j}.
$$

Using Vitali's covering lemma, we can extract a subfamily of disjoint cubes
 $Q_i$ such that
$$
F \subset \bigcup_{i} 3Q_i.
$$

By homogeneity,
$$
1<
\Bigl\|f_1\prod_{j=2}^m \bigl((f_j^{p_0})_{Q_i}\bigr)^{1/p_0}\Bigr\|_{\Phi,Q_i}
$$

and by the properties of the norm $\|\cdot\|_{\Phi, Q_i}$, this is the same as
$$
1<\frac{1}{|Q_i|} \int_{Q_i}
\Phi\Bigl(f_1(y)\prod_{j=2}^m \bigl((f_j^{p_0})_{Q_i}\bigr)^{1/p_0}\Bigr)dy.
$$
Using the fact that $\Phi$ is submultiplicative, it yields that
$$
1< \frac{1}{|Q_i|}\int_{Q_i}\Phi\bigl(f_1\bigr)dy
\prod_{j=2}^m \Phi\bigl(\bigl((f_j^{p_0})_{Q_i}\bigr)^{1/p_0}\bigr).
$$
Let $\Phi_0(t)=t(1+\log^+ t)^{p_0}$. By the Jensen inequality, we have
\begin{align*}
\Phi\bigl(\bigl((f_j^{p_0})_{Q_i}\bigr)^{1/p_0}\bigr)
&=\frac{1}{|Q_i|}\int_{Q_i}|f_j(y)|^{p_0}dy\Bigl(
1+\frac{1}{p_0}\log^+\frac{1}{|Q_i|}\int_{Q_i}|f_j(y)|^{p_0}dy\Bigr)^{p_0}
\\
&\le \frac{1}{|Q_i|}\int_{Q_i}|f_j(y)|^{p_0}dy
\Bigl(1+\log^+\frac{1}{|Q_i|}\int_{Q_i}|f_j(y)|^{p_0}dy\Bigr)^{p_0}
\\
&=\Phi_0\Bigl(\frac{1}{|Q_i|}\int_{Q_i}|f_j(y)|^{p_0}dy\Bigr)\\&
\le\frac{1}{|Q_i|}\int_{Q_i}\Phi_0(|f_j|^{p_0})(y)dy
\\
&=\frac{1}{|Q_i|}\int_{Q_i}|f_j(y)|^{p_0}
\Bigl(1+\log^+|f_j(y)|^{p_0}dy\Bigr)^{p_0}dy
\\
&\le\frac{ p_0^{p_0}}{|Q_i|}\int_{Q_i}|f_j(y)|^{p_0}
(1+\log^+|f_j(y)|)^{p_0}dy\\&
=\frac{p_0^{p_0}}{|Q_i|}\int_{Q_i}\Phi(|f_j|)(y)dy.
\end{align*}
Finally, by the condition assumed on the weights and the H\"{o}lder inequality at, one obtains 
discrete level,
\begin{equation*}
\begin{split}
\nu_{\vec{\omega}}(F)^m\approx
\Bigl(\sum_{i} \nu_{\vec{\omega}}(Q_i)\Bigr)^m
\leq p_0^{p_0}
\Bigl(\sum_{i} \prod_{j=1}^m \inf_Q \omega_j^{1/m}|Q_i|^{1/ m}\Bigl(
\frac{1}{|Q_i|} \int_{Q_i} \Phi\bigl(f_j\bigr)dy\Bigr)^{1/ m}\Bigr)^m
\\
\leq  p_0^{p_0}
\Bigl(\sum_{i} \prod_{j=1}^m\Bigr( \int_{Q_i} \Phi\bigl(f_j(y)\bigr)
\omega_j(y)dy\Bigr)^{1/m}\Bigr)^m
\leq  p_0^{p_0}\prod_{j=1}^m \Bigl( \int_{\mathbb{R}^n}
\Phi \left(|f_j(x)|\right)\omega_j(x)dx\Bigr).\\
\end{split}
\end{equation*}
which concludes the proof of Lemma \ref{5}.
\end{proof}
Using the above lemmas, now, we can show Theorem \ref{3}. \par
\textbf{Proof of Theorem \ref{3}.}  It is enough to prove the result for the
operator $T_{b}$ defined in \eqref{41}.
 By homogeneity we may assume $t=1$.
Since $\Phi$ is submultiplicative, Lemma \ref{4} and Lemma \ref{5} yield that
\begin{equation}\label{555}
\begin{split}
\nu_{\vec{\omega}}&\bigg(\bigg\{x\in \mathbb{R}^n:
T_{b}\vec{f}(x)>1 \bigg\}\bigg)^m\\&\leq
C\sup_{t>0}\frac{1}{\Phi (1/t)^m}
\nu_{\vec{\omega}}\bigg(\bigg\{x\in \mathbb{R}^n:
T_{b}\vec{f}(x)>t^m
\bigg\}\bigg)^m\\
&\leq C\sup_{t>0}\frac{1}{\Phi (1/t)^m} \nu_{\vec{\omega}}\bigg(\bigg\{x\in \mathbb{R}^n: \mathcal{M}_{\Phi}^{(1)}\vec{f}(x)>t^m
\bigg\}\bigg)^m\\
&\leq C\sup_{t>0}\frac{1}{\Phi (1/t)^m}\prod_{j=1}^m
\int_{\mathbb{R}^n}\Phi(\frac{|f_j(x)|}{t})
\omega_j(x)dx\\
&\leq C\sup_{t>0}\frac{1}{\Phi (1/t)^m}\prod_{j=1}^m
\int_{\mathbb{R}^n}\Phi(|f_j(x)|)\Phi (1/t)
\omega_j(x)dx\\
&\leq C \prod_{j=1}^m
\int_{\mathbb{R}^n}\Phi (|f_j(x)|)
\omega_j(x)dx.
\end{split}
\end{equation}
This complete the proof of Theorem \ref{3}.

\end{document}